\documentclass[12pt,reqno]{amsart}
\usepackage{stmaryrd}
\usepackage[margin=1in]{geometry}
\usepackage[colorlinks,linkcolor=magenta,citecolor=blue]{hyperref}

\usepackage{amssymb}
\usepackage{amsfonts}
\usepackage{amsmath}
\usepackage{amsthm}
\usepackage{color}
\usepackage{graphicx}
\usepackage{epsfig,mathrsfs}
\usepackage[bf,SL,BF]{subfigure}
\usepackage{fancyhdr}
\usepackage{subfigure}
\usepackage{caption}
\usepackage{wrapfig}
\usepackage{cases}
\usepackage{makecell}
\usepackage{multirow}
\usepackage{algorithm}
\usepackage{enumitem}
\usepackage{algorithmic}
\usepackage{comment}

\newtheorem{theorem}{Theorem}
\newtheorem{lemma}[theorem]{Lemma}

\newtheorem{definition}[theorem]{Definition}

\newtheorem{remark}{Remark}

\numberwithin{equation}{section}

\numberwithin{theorem}{section}

\numberwithin{subsection}{section}

\newcommand{\intQT}{\int_0^T\int_{\Omega}}
\newcommand{\R}{\mathbb{R}}
\newcommand{\pa}{\partial}
\newcommand{\LO}[1]{L^{#1}(\Omega)}
\newcommand{\LQ}[1]{L^{#1}(Q_T)}

\newcommand{\red}{\textcolor{black}}
\newcommand{\bra}[1]{\left(#1\right)}
\newcommand{\sumi}{\sum_{i=1}^{m}}
\newcommand{\intO}{\int_{\Omega}}
\newcommand{\eps}{\varepsilon}

\newcommand{\vt}{\vartheta}

\title[Quasi-linear reaction-diffusion systems]{On quasi-linear reaction diffusion systems\\ arising from compartmental SEIR models}

\author[J. Yang]{Juan Yang}
\address{Juan Yang \hfill\break
    School of Mathematics and Statistics, Lanzhou University,
    Lanzhou, 730000, PR China,
    \hfill\break
	Institute of Mathematics and Scientific Computing, University of Graz,
	Heinrichstrasse 36, 8010 Graz, Austria}
\email{yangjuan18@lzu.edu.cn, yangjuan0912@gmail.com}

\author[J. Morgan]{Jeff Morgan}
\address{Jeff Morgan \hfill\break
	Department of Mathematics,
	University of Houston, Houston, Texas 77204, USA}
\email{jmorgan@math.uh.edu}

\author[B.Q. Tang]{Bao Quoc Tang}
\address{Bao Quoc Tang \hfill\break
	Institute of Mathematics and Scientific Computing, University of Graz,
	Heinrichstrasse 36, 8010 Graz, Austria}
\email{quoc.tang@uni-graz.at, baotangquoc@gmail.com}

\begin{document}

\begin{abstract}
	
	The global existence and boundedness of solutions to quasi-linear reaction-diffusion systems are investigated. The system arises from compartmental models describing \red{the spread} of \red{infectious diseases} proposed in \cite{viguerie2021simulating,viguerie2020diffusion}, where the diffusion rate is assumed to depend on the total population, leading to quasilinear diffusion with possible degeneracy. The mathematical analysis of this model has been addressed recently in \cite{auricchio2023well} where it was essentially assumed that all sub-populations diffuse at the same rate, which \red{yields} a positive lower bound of the total population, thus removing the degeneracy. In this work, we remove this assumption completely and show the global existence and boundedness of solutions by exploiting a recently developed $L^p$-energy method. Our approach is applicable to a larger class of systems and is sufficiently robust to allow model variants \red{and different} boundary conditions.
\end{abstract}

\maketitle
\noindent{\small{\textbf{Classification AMS 2010:} 35A01, 35K57, 35K59, 35Q92}}


\tableofcontents

\section{Introduction}
\subsection{Mathematical analysis of quasilinear SEIRD models}

Modelling the spatial spread of infectious disease is a classical topic and has recently attracted much more interest from the mathematical community due to the Covid-19 epidemic. Classical models include \red{the} susceptible-infected-removed (SIR) system and its \red{variants}, such as susceptible-exposed-infected-removed (SEIR) or susceptible-exposed-infected-removed-deceased (SEIRD). To account \red{for} spatial heterogeneity, a compartmental model was proposed in \cite{viguerie2021simulating,viguerie2020diffusion} where the diffusion coefficients are highly heterogeneous and depend on the total population density, leading to a \textit{quasi-linear} reaction-diffusion system. The simulations therein showed a strong qualitative agreement from the forecast using the model and collected data in Lombardy.

\medskip

Let $\Omega \subset \mathbb R^N$, $N\ge 1$, be a bounded domain\footnote{The domain $\Omega$ was considered in $\R^2$ and $\R^3$ in \cite{viguerie2021simulating} and \cite{auricchio2023well}, respectively. For the current work, we can deal with the case of arbitrary dimension.} with smooth boundary $\partial\Omega$ \red{such that $\Omega$ lies locally on one side of $\pa\Omega$}. Let $s(x,t)$, $e(x,t)$, $i(x,t)$, $r(x,t)$, and $d(x,t)$ denote the susceptible, exposed, infected, recovered, and deceased population densities, respectively, at spatial position $x\in\Omega$ and at time $t>0$, and \red{let} $n(x,t):=s(x,t)+e(x,t)+i(x,t)+r(x,t)$ \red{be} the total living population density. The proposed model in \cite{viguerie2021simulating} reads as
\begin{equation}\label{sys-m}
  \begin{cases}
   \pa_t s= \alpha n - (1 - A_0/n) \beta_i si - (1 - A_0/n) \beta_e se - \mu s + \nabla\cdot(n \nu_s \nabla s)\\
   \pa_t e= (1 - A_0/n) \beta_i si + (1 - A_0/n) \beta_e se - \sigma e - \phi_e e - \mu e + \nabla\cdot(n \nu_e \nabla e)\\
   \pa_t i = \sigma e - \phi_d i - \phi_r i - \mu i + \nabla\cdot(n \nu_i \nabla i)\\
   \pa_t r = \phi_r i + \phi_e e - \mu r + \nabla\cdot(n \nu_r \nabla r)\\
   \pa_t d = \phi_d i \,,
  \end{cases}
\end{equation}
where the positive diffusion coefficients $\nu_s, \nu_e, \nu_i, \nu_r$, the birth rate $\alpha$, the inverse of the incubation period $\sigma$, the asymptomatic recovery rate $\phi_e$, the infected recovery rate $\phi_r$, the infected mortality rate $\phi_d$, the asymptomatic contact rate $\beta_e$, the symptomatic contact rate $\beta_i$, and the general mortality rate $\mu$ are positive and may depend on time and space.

The mathematical analysis of \eqref{sys-m} \red{was} first studied in \cite{auricchio2023well}. One \red{distinguishing characteristic} of this system, in comparison with many other variants of SIR models in the literature, is that the diffusion of all species depends on the total population density $n$, which makes \eqref{sys-m} a \textit{quasilinear} reaction-diffusion system. This also brings possible degeneracy to the diffusion operators and makes the problem more challenging. The global existence and boundedness of solutions to \eqref{sys-m} was shown in \cite{auricchio2023well} under the following assumptions and modifications:
\begin{itemize}
	\item[(i)] all diffusion rates are the same\footnote{In fact, \cite{auricchio2023well} generalised the diffusion rate to $\kappa(n)$ for a continuous function $\kappa(\cdot)$, but it does not alter the fact that the diffusion rates must be the same so that the technique therein can be applied.}, i.e. $\nu_s = \nu_e = \nu_i = \nu_r = \nu$,
	\item[(ii)] the terms $(1-A_0/n)$ and $\phi_di$ are replaced by a non-singular function $A(n)$, for instance $A(n) = (1-A_0/n)_+$, and $\phi_dni$, respectively.
\end{itemize}
The replacement of $(1-A_0/n)$ \red{avoids} the singularity when $n$ gets close to zero. A closer examination later reveals that this is not necessary since one can estimate
$(si)/n \le s$ and $(se)/n\le e$. The assumptions of the same diffusion rate and replacement of $\phi_di$ are, on the other hand, essential as they help to obtain a \textit{key property} that the total population density $n$ is bounded \textit{pointwise} from below by a positive constant for all time, provided the initial density $n_0$ is also bounded from below. Indeed, summing the equations of $s$, $e$, $i$, and $r$, keeping in mind the aforementioned modifications, yields an equation for $n$ of the form
\begin{equation*}
	\partial_tn - \nabla\cdot(\nu n\nabla n) = -\phi_din + (\alpha-\mu)n \red{\geq (\alpha-\mu)n-\phi_d n^2}
\end{equation*}
from which the lower bound \red{for $n$ on finite time intervals} follows from \red{the} lower bound of initial data. With this lower bound of $n$, the diffusion operators become non-degenerate and the analysis can be carried out in a standard way. \red{This strategy used in \cite{auricchio2023well} seems to break down   when the diffusion rates are different, but it still applies when the term $-\phi_d i$ stays in place, since in this case
\begin{equation*}
  \partial_tn - \nabla\cdot(\nu n\nabla n) = -\phi_di + (\alpha-\mu)n \geq (\alpha-\mu-\phi_d)n.
\end{equation*}
So $n$ is bounded from below on finite time intervals if $n(x,0)\geq\alpha>0$ for all $x\in\Omega$}.

\medskip
In this paper, we set out to remove the assumptions (i) and (ii) above and show the global existence and boundedness of solutions to the original system \eqref{sys-m}. In fact, we show the results for a much larger class of quasilinear reaction-diffusion systems which contains \eqref{sys-m} as a special case. Our key idea is to exploit a recently developed $L^p$-energy approach in e.g. \cite{morgan2021global,fitzgibbon2021reaction} which \red{does} not require any lower bound on $n$ except for its natural nonnegativity. In the next subsection, we provide the general setting while the main results and ideas are presented in subsection \ref{subsec1.3}.


\subsection{Problem setting}\label{subsec1.2}
Let $1\leq N\in\mathbb{N}$, and $\Omega\subset\mathbb{R}^N$ be a bounded domain with Lipschitz boundary $\pa \Omega$. Let $2\leq m\in\mathbb{N}$. In this paper, we study the global existence and boundedness of the following \textit{quasi-linear reaction-diffusion system} of concentrations $u=(u_1, \cdots, u_m)$, for any $i\in\{1,\cdots, m\}$,
\begin{equation}\label{system}
	\begin{cases}
		\pa_t u_i-\nabla\cdot(D_i(x,t)\Phi(u)\nabla u_i)=f_i(x,t,u),  &x\in \Omega, t>0\\
		(D_i(x,t)\Phi(u)\nabla u_i)\cdot \eta=0,   &x\in \pa \Omega, t>0\\
		u_i(x,0)=u_{i,0}(x),   &x\in \Omega,
	\end{cases}
\end{equation}
where \red{$\eta$} is the unit outward normal vector on $\pa\Omega$, initial date $u_{i,0}$ are bounded and non-negative,  the diffusion matrix $D_i$: $\Omega\times [0,\infty)\rightarrow\mathbb{R}^{N\times N}$ satisfies
\begin{equation}\label{D:cond-1}
	\lambda |\xi|^2\leq \xi^\top D_i(x,t)\, \xi, \quad \forall (x,t)\in \Omega\times[0,\infty), ~\forall \xi\in\mathbb{R}^N, ~\forall i=1,\cdots,m
\end{equation}
for some $\lambda>0$, and for each $T>0$,
\begin{equation}\label{D:cond-2}
	D_{i}\in L^{\infty}(\Omega\times(0,T); \R^{N\times N}), \quad\forall i=1,\cdots,m,
\end{equation}
and $\Phi(u)$ satisfies:
\begin{enumerate}[label=(Q\theenumi),ref=Q\theenumi]
	\item \label{Q1} $\Phi(u): \R^m_+\rightarrow \R_+$ is continuous;
	\item \label{Q2} There is some $b\geq0$ and $M>0$, such that
	\begin{equation*}
		\Phi(u)\geq M \,u_i^b, \quad \forall  u \in \mathbb R_+^m \text{ and } i=1,\cdots,m;
	\end{equation*}
	\item \label{Q3} There exist $\pi>0$ and $\widetilde{M}>0$ such that
	\begin{equation*}
		\Phi(u) \leq \widetilde{M} \bra{1+\sumi u_i^{\pi}}, \quad \forall  u\in \mathbb R_+^m.
	\end{equation*}
\end{enumerate}
The nonlinearities $f_i(x, t, u): \Omega\times\mathbb{R}_+\times\mathbb{R}^m_+\rightarrow\mathbb{R}$ satisfy the following conditions:
\begin{enumerate}[label=(A\theenumi),ref=A\theenumi]
	\item  \label{A1} For any $i=1,\cdots,m$ and any $(x,t)\in\Omega\times\mathbb{R}_+$, ~$f_i(x,t,\cdot): \mathbb{R}^m\rightarrow \mathbb{R}$  is locally Lipschitz continuous uniformly in $(x,t)\in\Omega\times(0,T)$ for any $T>0$;
	\item  \label{A2} For any $i=1,\cdots,m$ and any  $(x,t)\in\Omega\times\mathbb{R}_+$, ~$f_i(x,t,\cdot)$ is quasi-positive, i.e., $f_i(x,t,u)\geq0$ for all $u\in \R ^m_+$ with $u_i=0$ for all $i=1,\cdots,m$;
	\item  \label{A3} There exists $c_1,\cdots,c_m>0$ and $K_1,K_2\in\R$ such that
	\begin{equation*}
		\sumi c_if_i(x,t,u)\leq K_1\sumi u_i+K_2, \quad \forall (x,t,u)\in\Omega\times\R_+\times\R^m_+;
	\end{equation*}
	\item  \label{A4} There exist $K_3>0, r>0$, and a lower triangular matrix $A=(a_{ij})$ with positive diagonal entries, and nonnegative entries otherwise, such that, for any $i=1,\cdots,m$
	\begin{equation*}
		\sum^i_{j=1}a_{ij}f_j(x,t,u)\leq K_3 \bra{1+\sumi u_i^r}, \quad \forall (x,t,u)\in\Omega\times\R_+\times\R^m_+
	\end{equation*}
	(we call this assumption an \textit{intermediate sum of order $r$});
	\item  \label{A5} The nonlinearities are bounded above by a polynomial, i.e., there exists $l>0$ and $K_4>0$ such that
	\begin{equation*}
		f_i(x,t,u)\leq K_4 \bra{1+\sumi u_i^l},  \quad \forall (x,t,u)\in\Omega\times\R_+\times\R^m_+, ~\forall i=1,\cdots, m.
	\end{equation*}
\end{enumerate}

The local Lipschitz continuity \eqref{A1} of the nonlinearities implies the existence of a local solution to \eqref{system} on a maximal interval $[0, T_{\max})$.  The quasi-positivity assumption \eqref{A2} assures that the solution to \eqref{system} is non-negative (as long as it exists) if the initial data is non-negative,  which is a natural assumption since we consider here $u_i$ as concentrations or densities.  The assumption \eqref{A3} gives an upper bound on the total mass of the system.
Reaction-diffusion systems satisfying \eqref{A1}, \eqref{A2} and \eqref{A3} appear naturally in modeling many real life phenomena, ranging from chemistry, biology, ecology, or social sciences. Remarkably, these natural assumptions are not enough to ensure global existence of bounded solutions as it was pointed out by counterexamples in \cite{pierre2000blowup} and \cite{Pierre2023Examples} even for \textit{semilinear} systems, i.e. $\Phi(u)\equiv 1$ and $D_i(x,t)\equiv D_i$ in \eqref{system} where $D_i\ne D_j$ for some $i\ne j$. The study of global existence of semilinear systems, i.e. \eqref{system} with $\Phi(u)\equiv 1$, has made considerable progress in the last decade, see e.g. \cite{canizo2014improved,fellner2016Global,laamri2011global,Laamri2017Global,fellner2020global, fellner2019uniform, fitzgibbon2021reaction} and the survey \cite{Pierre2010Global}.

\medskip
One can readily check that all assumptions \eqref{Q1}--\eqref{Q3}, \eqref{A1}--\eqref{A5} are fulfilled for the SEIR system \eqref{sys-m} (excluding the equation of $d$ as it is uncoupled) provided that the diffusion rates are bounded from below by positive constants, and all other rates are nonnegative and bounded functions of $x$ and $t$. Indeed, by writing $u_1 = s$, $u_2 = e$, $u_3 = i$, $u_4 = r$, and $u = (u_1,\ldots,u_4)$, we have $\Phi(u) = u_1+u_2+u_3 + u_4$ satisfying \eqref{Q1}--\eqref{Q3} with $b = \pi = 1$. Now \eqref{A1}--\eqref{A3} and \eqref{A5} are obviously fulfilled. To check \eqref{A4}, we observe the first nonlinearity
\begin{align*}
	&\alpha n - (1 - A_0/n) \beta_i si - (1 - A_0/n) \beta_e se - \mu s\\
	&= \alpha n - \beta_i si + A_0\beta_i\frac{si}{n} - \beta_ese + A_0\beta_e \frac{se}{n} - \mu s\\
	&\le \alpha n  + (A_0\beta_i + A_0\beta_e)s
\end{align*}
where we used $si/n \le s$ and $se/n\le e$ and the nonnegativity of the rates. Thanks to the boundedness of the rates, we see that the first nonlinearity is bounded above by a linear combination of all components. It's immediate that the third and forth nonlinearities, as well as the sum of the first and second nonlinearities are bounded by a linear combination of all components, since all nonlinear terms are cancelled out when summing. Hence, by choosing $r = 1$ and the matrix
\begin{equation*}
	A = \begin{pmatrix}
		1 & 0 & 0 & 0\\
		1 & 1 & 0 & 0\\
		0 & 0 & 1 & 0\\
		0 & 0 & 0 & 1
	\end{pmatrix}
\end{equation*}
we see that \eqref{A4} is satisfied. It is noted that all these assumptions also hold if we replace $-\phi_di$ in \eqref{sys-m} by $-\phi_d in$ as done in \cite{auricchio2023well}. Therefore, \eqref{sys-m}, as well as the modified one in \cite{auricchio2023well}, is indeed a special case of the general system \eqref{system}.

\subsection{Main results}\label{subsec1.3}
Let us start with the first main result about the global existence and boundedness of system \eqref{system}.
\begin{theorem}\label{th1}
  Assume \eqref{D:cond-1}, \eqref{D:cond-2},  \eqref{Q1}, \eqref{Q2}, \eqref{Q3}, \eqref{A1}, \eqref{A2}, \eqref{A3}, \eqref{A5} and \eqref{A4} with
  \begin{equation}\label{growth-condition}
  1\leq r<1+b+\frac{2}{N}.
  \end{equation}
  Then for any nonnegative, bounded initial datum $u_0\in (L^{\infty}(\Omega))^m$, there exists a global weak solution to \eqref{system}
  with $u_i\in L^{\infty}_{loc}(0,\infty; L^{\infty}(\Omega))$ for all $i=1,\cdots,m.$  In particular, if $K_1<0$ or $K_1=K_2=0$, then the solution is bounded uniformly in time, i.e.
  \begin{equation}\label{Linftybound}
    \sup_{t\geq0}\|u_i(t)\|_{\LO{\infty}}<+\infty, \quad \forall i=1,\cdots,m.
  \end{equation}
\end{theorem}

\begin{remark}\hfill
\begin{itemize}
    \item When $b=0$, we have $\Phi(u)\geq M$. By letting $\Phi(u)=1$, our results cover the result in \cite[Theorem 1.1]{fitzgibbon2021reaction}.
	\item It is noted that under the assumption $K_1<0$ or $K_1 = K_2 = 0$, we can get the $L^1(\Omega)$-norm of the solution bounded uniformly in time (see Lemma \ref{L1-uni-bound}). This fact can be used to show the uniform-in-time bound \eqref{Linftybound}. \red{As a result,} if we can use other ideas or use other structure of the system to get the $L^1(\Omega)$-norm of the solution bound uniformly in time, we can remove the assumption $K_1<0$ or $K_1 = K_2 = 0$.
     \item It is noted that the key idea to prove global existence is to construct the $L^p$-energy function. Actually, the intermediate sum condition \eqref{A4} is essential to construct the $L^p$-energy function.
\end{itemize}
\end{remark}

\medskip
The proof of Theorem \ref{th1} is based on an $L^p$-energy approach. The \red{traditional} idea \red{has been} to construct an energy function that is decreasing or at least bounded in time for the \eqref{system} of the form
\begin{equation*}
\mathscr{L}[u] = \sumi \intO h_i(u_i) \,dx.
\end{equation*}
It is noted that if $h_i(z) \sim z^p$, this yields an $L^p$-estimate of the solution, and for $p$ large enough, using \red{a bootstrapping procedure, one obtains an} $L^\infty$-estimate which implies global existence. Unfortunately, this approach is very likely to fail under the general assumptions \eqref{A3}-\eqref{A4}, except for some very special cases.
The {\it duality method} (see e.g. \cite{Hollis1987Global,morgan1989global,morgan1990boundedness, Pierre2010Global,canizo2014improved,morgan2020boundedness}) is very efficient when dealing with systems with {\it constant or smooth diffusion coefficients}. Using this method, one gets an $L^{2+\eps}(\Omega\times(0,T))$-estimate from the mass control condition \eqref{A3}.
Combining this initial estimate, another duality argument (see \cite{morgan2020boundedness}) and bootstrap argument, using the intermediate sum condition \eqref{A4} and the growth assumption \eqref{A5}, one can obtain an  $L^\infty(\Omega\times(0, T))$-estimate that ensures global existence (more details can be found in \cite{morgan2020boundedness}).
 \red{It is noted that this method does not extend} to the case of merely bounded measurable or quasi-linear diffusion coefficients, unless some additional regularity assumptions are given (see \cite{desvillettes2007global,bothe2017global}). Recently, an $L^p$-energy approach \red{was} given in \cite{fitzgibbon2021reaction,morgan2021global},
whose preliminary ideas have been used previously in, e.g., \cite{malham1998global,kouachi2001existence}. It is remarked that a similar method has also been successfully applied to cross-diffusion systems, see \cite{laurenccot2022bounded}. The core idea of this $L^p$-energy method is to construct a generalized $L^p$-energy function of the following form:
\begin{equation*}
  \mathscr{L}_p[u](t)=\int_{\Omega} \mathscr{H}_p[u](t) \,d x,
\end{equation*}
where \red{$p\in \mathbb{N}$ with $p\geq 2$ and}
\begin{equation*}
\mathscr{H}_p[u](t)=\sum_{\beta \in \mathbb{Z}_{+}^m,|\beta|=p}\left(\begin{array}{l}
p \\
\beta
\end{array}\right) \theta^{\beta^2} u(t)^\beta
\end{equation*}
with
\begin{equation*}
u(t)^\beta = \prod_{i=1}^{m}u_i^{\beta_i}, \theta^{\beta^2}= \prod_{i=1}^{m}\theta_i^{\beta^2_i} \text{ and }
\left(\begin{array}{l}
p \\
\beta
\end{array}\right)=\frac{p !}{\beta_{1} ! \cdots \beta_{m} !},
\end{equation*}
where $\theta=\left(\theta_1, \ldots, \theta_m\right)$ and $\theta_1, \ldots, \theta_m$ are positive real numbers which will be determined later. The function $\mathscr L_p[u]$ contains all (mixed) multi-variable polynomials of order $p$ with carefully chosen coefficients. Thanks to the non-negativity of the solution, $(\mathscr{L}_p[u])^{1/p}$ is an equivalent $L^p$-estimate of $u$. Certainly, the difficulty is to choose these coefficients so that they are compatible with both diffusion and reactions. In our case, the assumptions on the quasilinear diffusion \eqref{Q2} and the intermediate sum condition \eqref{A4} allow us construct such an energy function.

\medskip
Comparing to \cite{fitzgibbon2021reaction,morgan2020boundedness}, this work extends this $L^p$-energy method to the case of {\it degenerate quasi-linear diffusion coefficients}.  Moreover, it is also shown that this method is sufficiently robust to model variants and different boundary conditions.

\medskip
Conditionally, if we can, by using a specific structure, \red{get a better a prior estimate} of solutions, the exponent $r$ in the intermediate sum condition \eqref{A4} can be enlarged. This is contained in the following theorem.

\begin{theorem}\label{th2}
  Assume \eqref{D:cond-1}, \eqref{D:cond-2},  \eqref{Q1}, \eqref{Q2}, \eqref{Q3}, \eqref{A1}, \eqref{A2}, \eqref{A4} and \eqref{A5}.
  Suppose that there exists a constant $a\geq 1$ such that
  \begin{equation}\label{condi-u-1}
    \|u_i\|_{L^{\infty}(0,T;\LO{a})}\leq \mathscr{F}(T), \quad {\forall i=1,\cdots,m}
  \end{equation}
  and
  \begin{equation}\label{condi-growth1}
    1\leq r<1+b+\frac{2a}{N}.
  \end{equation}
  Then for any nonnegative, bounded initial datum $u_0\in (L^{\infty}(\Omega))^m$, there exists a global weak solution to \eqref{system}
  with $u_i\in L^{\infty}_{loc}(0,\infty; L^{\infty}(\Omega))$ for all $i=1,\cdots,m.$  In particular, if $\sup_{T\geq0}\mathscr{F}(T)<\infty$, then the solution is bounded uniformly in time, i.e.
  \begin{equation}
    \sup_{t\geq0}\|u_i(t)\|_{\LO{\infty}}<+\infty, \quad \forall i=1,\cdots,m.
  \end{equation}
\end{theorem}
\begin{remark}
  When $b=0$,  our results cover the result in \cite[Theorem 1.3]{fitzgibbon2021reaction}.
\end{remark}

\begin{theorem}\label{th3}
  Assume \eqref{D:cond-1}, \eqref{D:cond-2},  \eqref{Q1}, \eqref{Q2}, \eqref{Q3}, \eqref{A1}, \eqref{A2}, \eqref{A4} and \eqref{A5}.
  Suppose that there exists a constant $q> 1$ such that
  \begin{equation}\label{condi-u-2}
    \|u_i\|_{L^{q}(0,T;\LO{q})}\leq \mathscr{F}(T), \quad {\forall i=1,\cdots,m}
  \end{equation}
  and
  \begin{equation}\label{condi-growth2}
   1\leq r< 1 + \frac{N}{N+2}b + \frac{2q}{N+2}.
  \end{equation}
  Then for any nonnegative, bounded initial datum $u_0\in (L^{\infty}(\Omega))^m$, there exists a global weak solution to \eqref{system}
  with $u_i\in L^{\infty}_{loc}(0,\infty; L^{\infty}(\Omega))$ for all $i=1,\cdots,m.$  In particular, if $\sup_{T\geq0}\mathscr{F}(T)<\infty$, then the solution is bounded uniformly in time, i.e.
  \begin{equation}
    \sup_{t\geq0}\|u_i(t)\|_{\LO{\infty}}<+\infty, \quad \forall i=1,\cdots,m.
  \end{equation}
\end{theorem}

\begin{remark}
  When $b=0$,  our results cover the result in \cite[Theorem 1.3]{fitzgibbon2021reaction}.
\end{remark}

\medskip
It is worth noting that other conditions can also lead to \red{a prior estimates}. For example, the entropy condition has been considered in many papers, especially when it involves chemical reactions, see e.g. \cite{caputo2009global,caputo2019solutions,souplet2018global,tang2018global}. This condition means that there exist the scalars $\mu_i\in\R$ such that

\begin{equation*}
  \sumi f_i(x,t,u)(\log u_i+\mu_i)\leq 0, \quad \forall (x,t,u)\in \Omega\times\R_+\times \R^m_+.
\end{equation*}
This condition guarantees an $\LO1$-estimate for $u_i\log u_i$.
Actually, this condition guarantees an $\LO1$-estimate for $H(u(\cdot,t))$, where
\begin{equation*}
  H(u)=\sumi u_i(\log u_i-1+\mu_i), \quad \forall u_i\geq 0.
\end{equation*}

Moreover, we can assume there exists a set $M=\prod_{k=1}^m(\alpha_i,\beta_i)$, where $\alpha_i,\beta_i$ are extended real numbers such that $\alpha_i<\beta_i$ for each $i=1,\ldots, m$, and solutions to (\ref{system}) remain in $M$ if initial data lies in $M$, a function $H:M\to \mathbb{R}_+$ that is $C^2$ and has the form
\[
H(u)=\sumi h_i(u_i),
\]
where $h_i:(\alpha_i,\beta_i)\to\mathbb{R}_+$ satisfies
\begin{equation}\label{H1}\tag{H1}
\begin{gathered}
h_i''(z) \ge 0, \quad \forall z\in (\alpha_i,\beta_i);\\
h_i(z) \text{ is bounded implies } z \text{ is bounded};\\
\nabla H(u)\cdot F(x,t,u)\le K_5\sumi h_i(u_i)+K_6, \quad \forall (x,t,u) \in\Omega\times\mathbb{R}_+\times\mathbb{R}_+^m
\end{gathered}
\end{equation}
for some $K_5, K_6>0$.
Then, we would obtain an $L^1(\Omega)$-estimate for $H(u(\cdot,t))$ (this has been introduced in \cite{morgan1989global}). In addition, intermediate sum conditions could also be written in the form
\begin{equation}\label{H2}\tag{H2}
A\begin{pmatrix}h_1'(u_1)F_1(x,t,u)\\\vdots\\h_m'(u_1)F_m(x,t,u)\end{pmatrix}\le K_7\vec{1}\left(\sumi h_i(u_i)+1\right)^r, \quad \forall(x,t,u)\in\Omega\times\mathbb{R}_+\times\mathbb{R}_+^m,
\end{equation}
that would lead to results in the same manner as we obtain from (\ref{A3}) and (\ref{A4}) above.

\begin{theorem}\label{th4}
  Assume \eqref{D:cond-1}, \eqref{D:cond-2},  \eqref{Q1}, \eqref{Q2}, \eqref{Q3}, \eqref{A1}, \eqref{A2}, \eqref{A5}, \eqref{H1} and \eqref{H2}. Assume moreover that
  \begin{equation}\label{growth-condition-1}
  1\leq r<1+b+\frac{2}{N}.
  \end{equation}
  Then for any nonnegative, bounded initial datum $u_0\in (L^{\infty}(\Omega))^m$, there exists a global weak solution to \eqref{system}
  with $u_i\in L^{\infty}_{loc}(0,\infty; L^{\infty}(\Omega))$ for all $i=1,\cdots,m.$  In particular, if $K_5<0$ or $K_5=K_6=0$, then the solution is bounded uniformly in time, i.e.
  \begin{equation}
    \sup_{t\geq0}\|u_i(t)\|_{\LO{\infty}}<+\infty, \quad \forall i=1,\cdots,m.
  \end{equation}
\end{theorem}

\begin{remark}
  When $b=0$,  our results cover the result in \cite[Theorem 1.2]{fitzgibbon2021reaction}.
\end{remark}

\medskip
Our method is sufficiently robust to extend to other boundary conditions (for example, Robin-type boundary conditions). The precise results are given in the following theorem.
\begin{theorem}\label{th5}
	Consider the system
	\begin{equation}\label{Sys_other_bc}
	\begin{cases}
		\pa_t u_i - \nabla\cdot(D_i(x,t)\Phi(u)\nabla u_i) = f_i(x,t,u), &x\in\Omega, t>0,\\
		D_i(x,t)\Phi(u)\nabla u_i(x,t)\cdot \eta + \alpha_i u_i(x,t) = 0, &x\in\pa\Omega, t>0,\\
		u_{i}(x,0) = u_{i,0}(x), &x\in\Omega,
	\end{cases}
	\end{equation}
	where $\eta$ is the unit  outward normal vector on $\pa\Omega$, and $\alpha_i\geq 0$ for all $i=1,\ldots, m$.
	
	Assume \eqref{D:cond-1}, \eqref{D:cond-2}, \eqref{Q1}, \eqref{Q2}, \eqref{Q3}, \eqref{A1}, \eqref{A2}, \eqref{A4}, \eqref{A5}. Moreover, assume either \eqref{A3} or \eqref{condi-u-1} or \eqref{condi-u-2} with
	\begin{equation*}
		0\leq r< \begin{cases}
			1 + b + \frac{2}{N}, &\text{ in case of } \eqref{A3},\\
			1 + b  + \frac{2a}{N}, &\text{ in case of } \eqref{condi-u-1},\\
			1 + \frac{N}{N+2}b + \frac{2q}{N+2}, &\text{ in case of } \eqref{condi-u-2}.
		\end{cases}
	\end{equation*}
	Then for any non-negative, bounded initial data $u_0\in \LO{\infty}^m$, there exists a  global weak solution to \eqref{Sys_other_bc} (see Definition \ref{def_diff_bc}) with $u_i\in L^{\infty}_{loc}(0,\infty;\LO{\infty})$ for all $i=1,\cdots, m$.
In particular, if $K_1< 0$ or $K_1= K_2 = 0$ in case of \eqref{A3}, or $\sup_{T\geq0}\mathscr{F}(T)<\infty$ in case of \eqref{condi-u-1} or \eqref{condi-u-2}, then the solution is bounded uniformly in time, i.e.
	\begin{equation*}
	  \sup_{t\geq0}\|u_i(t)\|_{\LO{\infty}}<+\infty, \quad \forall i=1,\cdots,m.
	\end{equation*}
\end{theorem}

\begin{remark}
	We believe Theorem \ref{th5} can also be extended to the semilinear boundary conditions of the form
	\begin{equation*}
		D_i(x,t)\nabla u_i(x,t)\cdot \eta + \alpha_i u_i(x,t) = G_i(u), \quad x\in\pa\Omega,
	\end{equation*}
	where the nonlinearities $G_i$ also satisfy a quasi-positivity condition and an intermediate sum condition. The details are left for the interested reader. We refer to \cite{morgan2021global,sharma2021global} for a related work dealing with constant diffusion coefficients.
\end{remark}

\medskip
\noindent{\bf The paper is organised as follows.}
In Section \ref{sub-existence}, we prove the existence of Theorem \ref{th1} by first considering an approximate system where we regularize the nonlinearities to obtain global approximate weak solutions. Then we derive uniform estimates, applying the key idea of the $L^p$-energy functions, and pass to the limit to obtain global existence of \eqref{system}. The uniform-in-time boundedness is proved in Section \ref{sub-uni}. The proofs of extended \red{Theorems} \ref{th2}, \ref{th3}, \ref{th4}, and \ref{th5} are given in Section \ref{sub-proof}. In Section \ref{sec-appli}, we show application of our results to a \red{Susceptible-Exposed-Infected-Recovered (SEIRD) model and its variants. Finally, \red{in Appendix \ref{appendix}, we give } the proof of Lemma \ref{lem-L-infty-bound}.}

\medskip
\noindent\textbf{Notation.} In this paper we use the following notation, some of which will be recalled from time to time:
\begin{itemize}
	\item For $T>0$ and $p\in [1,\infty]$, $Q_T:= \Omega\times(0,T)$ and
	\begin{equation*}
		L^p(Q_T):= L^p(0,T;L^p(\Omega))
	\end{equation*}
	equipped with the usual norm
	\begin{equation*}
		\|f\|_{\LQ{p}}:= \bra{\int_0^T\intO |f|^pdxdt}^{1/p}
	\end{equation*}
	for $1\leq p < \infty$ and
	\begin{equation*}
		\|f\|_{\LQ{\infty}}:= \underset{(x,t)\in Q_T}{\text{\normalfont ess sup}}|f(x,t)|.
	\end{equation*}
	\item For  $p\in [1,\infty]$, $\tau \geq 0$ and $\delta>0$, we denote by
	\begin{equation*}
		Q_{\tau,\tau+\delta}:= \Omega\times(\tau,\tau+\delta)
	\end{equation*}
	and
	\begin{equation*}
		L^p(Q_{\tau,\tau+\delta}):= L^p(\tau,\tau+\delta; L^p(\Omega)).
	\end{equation*}
\end{itemize}

\section{Global existence of bounded weak solutions}\label{sub-existence}
\begin{definition}[Weak solutions]\label{def}
	A vector of nonnegative state variables $u=(u_1,\cdots,u_m)$ is called a weak solution to \eqref{system} on $(0,T)$ if
	$$\pa_tu_i\in L^2(0,T;(H^1(\Omega))'), \quad \Phi(u)\nabla u_i\in L^2(0,T; L^2(\Omega)) $$
	and
	$$\nabla (u_i^{\frac{b+2}{2}})\in L^2(0,T; L^2(\Omega)) $$
	with $u_i(\cdot, 0)=u_{i,0}(\cdot)$ for all $i=1,\cdots,m$, and  for any test function $ \red{\psi} \in L^2(0,T; H^1(\Omega))$
	we have
	\begin{equation*}
		\begin{split}
			\int^T_0 \intO \pa_t u_i  \psi\,dxdt+\int^T_0 \intO D_i(x,t)\Phi(u)\nabla u_i \cdot \nabla \psi \,dxdt
			=\int^T_0 \intO f_i(u)\psi \,dxdt.
		\end{split}
	\end{equation*}
\end{definition}

\subsection{Approximate system}\label{subsec-appro}

For $i=1,\cdots,m$ and $\eps>0$, consider the following approximating system for $u^{\eps}=(u_1^{\eps}, \cdots, u_m^{\eps})$,
\begin{equation}\label{appro-system}
\begin{cases}
   \pa_t u_i^{\eps}-\nabla\cdot(D_i(x,t)\Phi(u^{\eps})\nabla u_i^{\eps})=f^{\eps}_i(u^{\eps}),  &x\in \Omega, t>0,\\
  (D_i(x,t) \Phi(u^{\eps}) \nabla u^{\eps}_i)\cdot \eta=0,   &x\in \pa \Omega, t>0,\\
  u^{\eps}_i(x,0)=u_{i,0}(x)+\eps,   &x\in \Omega,
\end{cases}
\end{equation}
where
\begin{equation*}
f^{\eps}_i(u^{\eps}):=\frac{f_i(u^{\eps})}{1+\eps \sum^m_{j=1}|f_j(u^{\eps})|}
\end{equation*}
and $u_{i,0}^{\eps}\in \LO{\infty}$ is non-negative and $u_{i,0}^{\eps}\stackrel{\eps\rightarrow0}{\longrightarrow}u_{i,0}$ in $\LO{\infty}$. With this approximation, it is easy to check that the approximated non-linearities $f_i^{\eps}$ still satisfy the assumptions \eqref{A1}-\eqref{A5}.

\begin{lemma}
  For any fixed $\eps>0$, there exists a global bounded, nonnegative weak solution to \eqref{appro-system} on any finite time interval $(0,T)$, $T>0$.
\end{lemma}

\begin{proof}
  Since for fixed $\eps>0$, the nonlinearities $f_i^{\eps}(u^{\eps})$ are Lipschitz continuous and bounded, i.e.,
  \begin{equation*}
f^{\eps}_i(u^{\eps}):=\frac{f_i(u^{\eps})}{1+\eps \sum^m_{j=1}|f_j(u^{\eps})|}\leq \frac{1}{\eps}, \quad \forall (x,t)\in Q_T.
\end{equation*}
The global existence of a weak solution of \eqref{appro-system} is standard (Using Galerkin approach).

Next we prove the nonnegativity of $u^{\eps}$. Denote $u_{i,+}^{\eps}:=\max\{u_i^{\eps},0\}$ and $u_{i,-}^{\eps}:=\min\{u_i^{\eps},0\}$. We consider the auxiliary system of \eqref{appro-system}
\begin{equation*}
  \pa_t u_i^{\eps}-\nabla\cdot(D_i(x,t)\Phi(u^{\eps})\nabla u_i^{\eps})=f^{\eps}_i(u_+^{\eps}),
\end{equation*}
where $u_+^{\eps}=(u_{i,+}^{\eps})_{i=1,\cdots,m}$.

By multiplying this auxiliary system by $u_{i,-}^{\eps}$ and using the quasi-positivity assumption \eqref{A2} (recall that this property also holds for $f_i^{\eps}$), we obtain
\begin{equation*}
\frac{1}{2}\intO |u_{i,-}^{\eps}|^2 \,dx+\lambda\intO \Phi(u^{\eps}) |\nabla u_{i,-}^{\eps}|^2\,dx \leq 0,
\end{equation*}
where we use $u_{i,0,-}^{\eps}=0$. Thus, we get $u_{i,-}^{\eps}=0$ a.e. in $Q_T$, this shows the desired nonnegativity.
\end{proof}

\subsection{Uniform-in-$\eps$ estimate}\label{subsec-uni-eps}
In this subsection, we prove crucial uniform-in-$\eps$ estimates for the solution to \eqref{appro-system}. Moreover, we want to emphasize that all the constants in this subsection are independent of $\eps$.

The following bound in $L^{\infty}(0,T; L^1(\Omega))$ is immediate.

\begin{lemma}\label{L1-bound}
Assume \eqref{A1}, \eqref{A2} and \eqref{A3}. Then  for any $T>0$, there exists a constant $M_T$ depending on $T, \Omega, \|u_{i,0}\|_{L^1(\Omega)}$ and $c_1,\cdots,c_m$, $K_1, K_2$ in \eqref{A3} such that
\begin{equation*}
  \sup_{t\in(0,T)}\|u_i^{\eps}(t)\|_{\LO1}\leq M_T, \quad \forall i=1,\cdots,m.
\end{equation*}
\end{lemma}

\begin{proof}
  By summing the equation in \eqref{appro-system}, integrating on $\Omega$ and using \eqref{A3} we have
  \begin{equation*}
    \frac{d}{dt}\sumi\intO c_iu_i^{\eps} \,dx \leq \intO \left(K_1 \sumi u_i^{\eps}(x,t) + K_2\right) \,dx.
  \end{equation*}
  The classical Gronwall inequality gives the desired estimate.
\end{proof}

\medskip
The following  $L^p$-energy function has been developed in \cite{morgan2021global,fitzgibbon2021reaction}.
We write $\mathbb{Z}_{+}^m$ as the set of all $m$ tuples of nonnegative integers. Addition and scalar multiplication by nonnegative integers of elements in $\mathbb{Z}_{+}^m$ is understood in the usual manner. If $\beta=\left(\beta_1, \ldots, \beta_m\right) \in \mathbb{Z}_{+}^m$ and \red{$p \in \mathbb{Z}_+$}, then we define $\beta^p=\left(\left(\beta_1\right)^p, \ldots,\left(\beta_m\right)^p\right)$. Also, if $\alpha=\left(\alpha_1, \ldots, \alpha_m\right) \in$ $\mathbb{Z}_{+}^m$, then we define $|\alpha|=\sum_{i=1}^m \alpha_i$. Finally, if $z=\left(z_1, \ldots, z_m\right) \in \mathbb{R}_{+}^m$ and $\alpha=$ $\left(\alpha_1, \ldots, \alpha_m\right) \in \mathbb{Z}_{+}^m$, then we define $z^\alpha=z_1^{\alpha_1} \cdots z_m^{\alpha_m}$, where we interpret $0^0$ to be 1. For any $2\leq p \in \mathbb N$, we build our $L^p$-energy function of the form
\begin{equation}\label{def-Lp-energy-fun}
  \mathscr{L}_p[u](t)=\int_{\Omega} \mathscr{H}_p[u](t) \,d x,
\end{equation}
where
\begin{equation}\label{def-H-p}
\mathscr{H}_p[u](t)=\sum_{\beta \in \mathbb{Z}_{+}^m,|\beta|=p}\left(\begin{array}{l}
p \\
\beta
\end{array}\right) \theta^{\beta^2} u(t)^\beta
\end{equation}
with
\begin{equation*}
\left(\begin{array}{l}
p \\
\beta
\end{array}\right)=\frac{p !}{\beta_{1} ! \cdots \beta_{m} !}
\end{equation*}
and $\theta=\left(\theta_1, \ldots, \theta_m\right)$, where $\theta_1, \ldots, \theta_m$ are positive real numbers which will be determined later. For $p=0,1,2$, one can write these functions explicitly as
\begin{equation*}
\mathscr{H}_0[u](t)=1 \text { and } \mathscr{H}_1[u](t)=\sum_{j=1}^m \theta_j u_j(t)
\end{equation*}
and
\begin{equation*}
\mathscr{H}_2[u](t)=\sum_{i=1}^m \theta_i^4 u_i(t)^2+2 \sum_{i=1}^{m-1} \sum_{j=i+1}^m \theta_i \theta_j u_i(t) u_j(t) .
\end{equation*}
Thanks to the nonnegativity of the solution, we have
\begin{equation*}
\mathscr{L}_p[u](t) \sim \sum_{i=1}^m\|u_i(t)\|_{L^p(\Omega)}^p ~\text{ for } p\geq 1.
\end{equation*}
\red{This will allow us} to use $\mathscr{L}_p[u](t)$ to obtain a priori estimates on $u$ for each $2 \leq p \in \mathbb{N}$. We will need two technical lemmas.
\begin{lemma}[\cite{fitzgibbon2021reaction}, Lemma 4.1]\label{technoque-Lem-1}
  Suppose $m \in \mathbb{N}, \theta=\left(\theta_1, \ldots, \theta_m\right)$, where $\theta_1, \ldots, \theta_m$ are positive real numbers, $\beta \in \mathbb{Z}_{+}^m$, and $\mathscr{H}_p[u]$ is defined in \eqref{def-H-p}. Then
$$
\frac{\partial}{\partial t} \mathscr{H}_0[u](t)=0, \quad \frac{\partial}{\partial t} \mathscr{H}_1[u](t)=\sum_{j=1}^m \theta_j \frac{\partial}{\partial t} u_j(t)
$$
and for $p \in \mathbb{N}$ such that $p \geq 2$,
$$
\frac{\partial}{\partial t} \mathscr{H}_p[u](t)=\sum_{|\beta|=p-1}\left(\begin{array}{l}
p \\
\beta
\end{array}\right) \theta^{\beta^2} u(t)^\beta \sum_{j=1}^m \theta_j^{2 \beta_j+1} \frac{\partial}{\partial t} u_j(t).
$$
\end{lemma}

\begin{lemma}[\cite{fitzgibbon2021reaction}, Lemma 4.2]\label{technoque-Lem-2}
Suppose $m \in \mathbb{N}, \theta=\left(\theta_1, \cdots, \theta_m\right)$, where $\theta_1, \cdots, \theta_m$ are positive real numbers, and let $\mathscr{H}_p[u]$ be defined in \eqref{def-H-p}. If $p \in \mathbb{N}$ such that $p \geq 2$, then
\begin{equation*}
\begin{aligned}
& \sum_{|\beta|=p-1}\left(\begin{array}{l}
p \\
\beta
\end{array}\right) \theta^{\beta^2} \sum_{k=1}^m \theta_k^{2 \beta_k+1}\left(A_k \nabla u_k\right) \cdot \nabla u^\beta \\
& =\sum_{|\beta|=p-2}\left(\begin{array}{c}
p \\
\beta
\end{array}\right) \theta^{\beta^2} u^{\beta} \sum_{k=1}^m \sum_{l=1}^m C_{k, l}(\beta)\left(A_k \nabla u_k\right) \cdot \nabla u_l,
\end{aligned}
\end{equation*}
where
\begin{equation*}
  C_{k, l}(\beta)=\left\{\begin{array}{cc}
\theta_k^{2 \beta_k+1} \theta_l^{2 \beta_l+1}, & k \neq \red{l}, \\
\theta_k^{4 \beta_k+4}, & k= \red{l}.
\end{array}\right.
\end{equation*}
\end{lemma}
Next, we need the following functional inequality.
\begin{lemma}\label{lem-func-ineq}
  Suppose $\Omega\subset \mathbb{R}^N$ such that the Gagliardo-Nirenberg inequality is satisfied and basic trace theorems apply (for instance $\Omega$ has a Lipschitz boundary). Let $a\geq 1$, $p\geq 2a$, $w: \overline{\Omega}\rightarrow\mathbb{R}_+$ such that $w^{\frac{b+p}{2}}\in W^{1,2}(\Omega)$ and there exists $K\geq0$ such that $\|w\|_{L^a(\Omega)}\leq K$. If $1\leq r<1+b+ \frac{2a}{N}$ and \red{$b\geq 0$}, then there exists $C_{\varepsilon}\geq 0$ (depending on $p,\varepsilon, r, a, b, K, \Omega,$ but independent of $w$) such that
  \begin{equation*}
    \int_{\Omega}w^{p-1+r} \,dx +\int_{\Omega}w^{b+p}\,dx \leq \varepsilon \left(\int_{\Omega} w^{p-2+b}|\nabla w|^{2}\,dx+\int_{\Omega} w^{b+p}\,dx\right)+C_{\varepsilon}.
  \end{equation*}
\end{lemma}

\begin{proof}
  By using Sobolev's embedding we have
  \begin{equation}\label{eq-L1-1}
  \begin{split}
    \int_{\Omega} &w^{p-2+b}|\nabla w|^{2}\,dx+\int_{\Omega} w^{b+p}\,dx\\
    &=\left(\frac{2}{b+p}\right)^{2}\int_{\Omega} |\nabla w^{\frac{b+p}{2}}|^{2}\,dx+\int_{\Omega} (w^{\frac{b+p}{2}})^{2}\,dx\\
    &\geq \min\left\{\left(\frac{2}{b+p}\right)^{2}, 1\right\}\|w^{\frac{b+p}{2}}\|^{2}_{H^{1}(\Omega)}.
  \end{split}
  \end{equation}
  Thus we have
  \begin{equation}\label{eq-l1}
    \int_{\Omega}w^{b+p}\,dx\leq \eps\left(\int_{\Omega} w^{p-2+b}|\nabla w|^{2}\,dx+ \int_{\Omega} w^{b+p}\,dx\right)+C(p,\varepsilon,b,K).
  \end{equation}
  Since
  \begin{equation*}
    \int_{\Omega}w^{p-1+r}\,dx=\|w^{\frac{p-1+r}{2}}\|^2_{\LO2},
  \end{equation*}
  if $r<b+1$. Then we have
  \begin{equation*}
    \int_{\Omega}w^{p-1+r} \,dx \leq \varepsilon \int_{\Omega} w^{b+p}\,dx +C_{\varepsilon}.
  \end{equation*}
  Next we consider $r\geq b+1$.
  Thanks to the Gagliardo-Nirenberg inequality, we have
  \begin{equation}\label{eq-L1-2}
  \begin{split}
    \int_{\Omega}w^{p-1+r}\,dx&=\|w^{\frac{b+p}{2}}\|_{L^{\frac{2(r+p-1)}{b+p}} (\Omega)}^{\frac{2(r+p-1)}{b+p}}\\
    &\leq C \|\nabla(w^{\frac{b+p}{2}})\|_{L^{2}(\Omega)}^{\frac{2(r+p-1)}{b+p}\alpha} \|w^{\frac{b+p}{2}}\|_{L^{1} (\Omega)}^{\frac{2(r+p-1)}{b+p}(1-\alpha)},
  \end{split}
  \end{equation}
  where
  \begin{equation*}
  \begin{split}
    \frac{b+p}{2(r+p-1)}=\alpha(\frac{1}{2}-\frac{1}{N})+1-\alpha.
  \end{split}
  \end{equation*}
It follows that
\begin{equation*}
  \begin{split}
    \alpha=\frac{N[2(r+p-1)-(b+p)]}{(r+p-1)(N+2)}
  \end{split}
\end{equation*}
and
\begin{equation*}
  \begin{split}
    1-\alpha=\frac{(r+p-1)(2-N)+N(b+p)}{(r+p-1)(N+2)}.
  \end{split}
\end{equation*}
Since $1\leq r<1+b+\frac{2a}{N}$, we find
\begin{equation*}
  \alpha\frac{2(r+p-1)}{b+p}<2.
\end{equation*}
We can  use Young's inequality to estimate
\begin{equation}\label{eq-L1-3}
  \begin{split}
    \|w^{\frac{b+p}{2}}\|_{L^{\frac{2(r+p-1)}{b+p}}(\Omega)}^{\frac{2(r+p-1)}{b+p}}
    &\leq \varepsilon \|\nabla(w^{\frac{b+p}{2}})\|_{L^{2}(\Omega)}^{2}+C_{\varepsilon}\|w^{\frac{b+p}{2}}\|_{L^{1} (\Omega)}^{\frac{(1-\alpha)2(r+p-1)}{b+p-(r+p-1)\alpha}},
  \end{split}
  \end{equation}
where
\begin{equation}\label{eq-L1-4}
  \begin{split}
    C_{\varepsilon}\|w^{\frac{b+p}{2}}\|_{L^{1} (\Omega)}^{\frac{(1-\alpha)2(r+p-1)}{b+p-(r+p-1)\alpha}}=C_{\varepsilon}\|w \|_{L^{\frac{b+p}{2}} (\Omega)}^{\frac{(1-\alpha)(b+p)(r+p-1)}{b+p-(r+p-1)\alpha}}.
  \end{split}
  \end{equation}
If $b+p=2a$, then this term is bounded by a constant depending on $K$, since $\|w\|_{L^a(\Omega)}\leq K$.
If $b+p>2a$, we use an interpolation inequality to have
\begin{equation}\label{eq-L1-5}
  \begin{split}
    C_{\varepsilon}\|w\|_{L^{\frac{b+p}{2}} (\Omega)}\leq C_{\varepsilon}\|w\|^{\theta}_{L^{p+r-1} (\Omega)}\|w\|^{1-\theta}_{L^{a} (\Omega)}\leq C_{\varepsilon, K} \|w\|^{\theta}_{L^{p+r-1} (\Omega)},
  \end{split}
\end{equation}
where $\theta\in(0,1)$ satisfies
$$\frac{2}{b+p}=\frac{\theta}{p+r-1}+\frac{1-\theta}{a}.$$
Note that
\begin{equation}\label{eq-L1-6}
\theta\frac{(1-\alpha)(b+p)(r+p-1)}{b+p-(r+p-1)\alpha}<p+r-1
\end{equation}
due to $r<b+1+ \frac{2a}{N}$ and $b\leq r-1<2r$.

From  \eqref{eq-L1-4}-\eqref{eq-L1-6} and Young's inequality, it follows that
\begin{equation}\label{eq-7}
  \begin{split}
    C_{\varepsilon}\|w \|^{\frac{(1-\alpha)(b+p)(r+p-1)}{b+p-(r+p-1)\alpha}}_{L^{\frac{b+p}{2}} (\Omega)}&\leq C_{\varepsilon, K}
    \|w \|^{\frac{(1-\alpha)(b+p)(r+p-1)}{b+p-(r+p-1)\alpha}\theta}_{L^{p+r-1} (\Omega)}\\
    &\leq \frac{1}{2} \| w \|^{p+r-1}_{L^{p+r-1}(\Omega)}+C(p,\varepsilon,b,K).
  \end{split}
\end{equation}
Inserting this into \eqref{eq-L1-2}, we get
\begin{equation*}
  \begin{split}
   \|w\|^{p+r-1}_{L^{p+r-1}(\Omega)}\leq 2\varepsilon \|w^{\frac{b+p}{2}}\|^{2}_{W^{1,2}(\Omega)} +C(p,\varepsilon,b,K).
  \end{split}
\end{equation*}
Combine this with \eqref{eq-L1-1} and \eqref{eq-l1} leads to the desired estimate
\begin{equation*}
    \int_{\Omega}w^{p-1+r} +\int_{\Omega}w^{b+p}\,dx \leq \varepsilon \left(\int_{\Omega} w^{p-2+b}|\nabla w|^{2}\,dx+\int_{\Omega} w^{b+p}\,dx\right)+C(p,\varepsilon,b,K).
  \end{equation*}
\end{proof}

\medskip
The following lemma shows results for the intermediate sum condition \eqref{A4}, which is crucial for constructing the $L^p$ energy function.
\begin{lemma}[\cite{fitzgibbon2021reaction}, Lemma 2.4]\label{lem-f-sum}
  Assume \eqref{A4}. Then there exist componentwise increasing functions $g_i: \mathbb{R}^{m-i} \rightarrow \mathbb{R}_{+}$ for $i=1, \ldots, m-1$ such that if  $\theta=\left(\theta_1, \ldots, \theta_m\right) \in(0, \infty)^m$ satisfies $\theta_m>0$ and $\theta_i \geq g_i\left(\theta_{i+1}, \ldots, \theta_m\right)$ for all $i=1, \ldots, m-1$, then
$$
\sum_{i=1}^m \theta_i f_i^{\eps}\left(x, t, u^{\varepsilon}\right) \leq K_\theta\left(1+\sum_{i=1}^m\left(u_i^{\varepsilon}\right)^r\right) \quad \forall\left(x, t, u^{\varepsilon}\right) \in \Omega \times \mathbb{R}_{+} \times \mathbb{R}_{+}^m
$$
for some constant $K_\theta$ depending on $\theta, g_i$, and $K_3$ in \eqref{A4}.
\end{lemma}

\medskip
We now use the $L^p$-energy functions  to obtain the $L^p$-estimates of $u^{\eps}$.
\begin{lemma}\label{lem-L-p-bound}
Assume \eqref{D:cond-1}, \eqref{D:cond-2},  \eqref{Q1}, \eqref{Q2}, \eqref{A1}, \eqref{A2},  \eqref{A3} and \eqref{A4} with $1\leq r<1+b+\frac{2}{N}$. Then for any $1 \leq p<\infty$ and any $T>0$, there exists a constant $C_{T, p}$ depending on $T, p$ and other parameters such that
$$
\sup _{t \in(0, T)}\|u_i^{\eps}(t)\|_{L^p(\Omega)} \leq C_{T, p} \quad \forall i=1, \ldots, m .
$$
 \end{lemma}
\begin{proof}
Let $u^{\eps}$ solve \eqref{appro-system}, and $\mathscr{L}_p(t):=\mathscr{L}_p\left[u^{\eps}\right](t)$ be defined in \eqref{def-Lp-energy-fun}. Then
\begin{equation*}
\begin{split}
\mathscr{L}_p^{\prime}(t)= & \int_{\Omega} \sum_{|\beta|=p-1}\left(\begin{array}{c}
p \\
\beta
\end{array}\right) \theta^{\beta^2} u^{\eps}(x, t)^\beta \sum_{k=1}^m \theta_k^{2 \beta_k+1} \frac{\partial}{\partial t} u_k^{\eps}(x, t)\, d x \\
= & \int_{\Omega} \sum_{|\beta|=p-1}\left(\begin{array}{c}
p \\
\beta
\end{array}\right) \theta^{\beta^2} u^{\eps}(x, t)^\beta \sum_{k=1}^m \theta_k^{2 \beta_k+1} \\
& \times\left[\nabla \cdot\left(D_k(x, t)\,\Phi(u^{\eps})\, \nabla u_k^{\eps}(x, t)\right)+f^{\eps}_k\left(u^{\eps}(x, t)\right)\right]\, d x\\
= & \int_{\Omega} \sum_{|\beta|=p-1}
\left(\begin{array}{c}
p \\
\beta
\end{array}\right)
\theta^{\beta^2} u^{\eps}(x, t)^\beta \sum_{k=1}^m \theta_k^{2 \beta_k+1}\nabla \cdot\left(D_k(x, t) \,\Phi(u^{\eps})\, \nabla u_k^{\eps}(x, t)\right)\, d x\\
&+ \int_{\Omega} \sum_{|\beta|=p-1}\left(\begin{array}{c}
p \\
\beta
\end{array}\right) \theta^{\beta^2} u^{\eps}(x, t)^\beta \sum_{k=1}^m \theta_k^{2 \beta_k+1} f^{\eps}_k\left(u^{\eps}(x, t)\right)\, d x\\
=:&(I)+(II).
\end{split}
\end{equation*}

For (I),  we apply Lemma \ref{technoque-Lem-2} and integration by parts, we have
\begin{equation*}
\begin{split}
(I):=& \int_{\Omega} \sum_{|\beta|=p-1}
\left(\begin{array}{c}
p \\
\beta
\end{array}\right)
\theta^{\beta^2} u^{\eps}(x, t)^\beta \sum_{k=1}^m \theta_k^{2 \beta_k+1}\nabla \cdot\left(D_k(x, t) \,\Phi(u^{\eps})\, \nabla u_k^{\eps}(x, t)\right)\, d x\\
=&-\int_{\Omega} \sum_{|\beta|=p-2}
\left(\begin{array}{l}
p \\
\beta
\end{array}\right)
\theta^{\beta^2} u^{\varepsilon}(x, t)^\beta \,\Phi(u^{\eps})\, \sum_{k=1}^m \sum_{l=1}^m C_{k, l}(\beta)\left(D_k(x,t) \nabla u_k^{\eps}\right) \cdot \nabla u_l^{\eps} \,dx
\end{split}
\end{equation*}
with
\begin{equation*}
C_{k, l}(\beta)=\left\{\begin{array}{cc}
\theta_k^{2 \beta_k+1} \theta_l^{2 \beta_l+1}, & k \neq l, \\
\theta_k^{4 \beta_{k}+4}, & k=l.
\end{array}\right.
\end{equation*}
For a given $\beta$ with $|\beta|=p-2$, create an $m N \times m N$ matrix $B(\beta)$ made up of $m^2$ blocks $B_{k, l}(\beta)$, each of size $N \times N$, where
$$
B_{k, l}(\beta)=\frac{1}{2} C_{k, l}(\beta)\left(D_k+D_l\right).
$$
Note that for each $k=1, \ldots, m$,
$$
B_{k, k}(\beta)=\theta_k^{4 \beta_k+4} D_k.
$$
Also,
$$
I=-\int_{\Omega} \sum_{|\beta|=p-2}\left(\begin{array}{c}
p \\
\beta
\end{array}\right) \theta^{\beta^2} u^{\eps}(x, t)^{\beta} \,\Phi(u^{\eps})\, \nabla u^{\eps}(x, t)^T B(\beta) \nabla u^{\eps}(x, t)\, dx,
$$
where $\nabla u^{\eps}(x, t)$ is a column vector of size $m N \times 1$, and for $j=1, \ldots, m$, entries $N(j-1)+1$ to $N j$ of $\nabla u^{\varepsilon}(x, t)$ are $\nabla u_j^{\eps}(x, t)$. We claim that if all of the entries in $\theta$ are sufficiently large, then $B(\beta)$ is positive definite. In fact, it is a simple matter to show it is positive definite if and only if the $m N \times m N$ matrix $\widetilde{B}(\beta)$ made up of $N \times N$ blocks
$$
\widetilde{B}_{k, l}(\beta)=\left\{\begin{array}{cc}
\theta_k^2 D_k, & k=l \\
\frac{1}{2}\left(D_k+D_l\right), & k \neq l
\end{array}\right.
$$
is positive definite. However, if we recall the uniform positive definiteness of the matrices $D_k$, we can show that if $\theta_i$ is sufficiently large for each $i$, then we have what we need. Consequently, returning to the above, we can show there exists $\alpha_p>0$ so that
\begin{equation}\label{eq-E}
  \begin{split}
    \mathscr{L}_p^{\prime}(t)+&\alpha_p \sum_{k=1}^m \int_{\Omega}u_k^{\eps}(x,t)^{b+p-2} \left|\nabla u^{\eps}_k(x, t)\right|^2 \,dx \\
    &\leq \int_{\Omega} \sum_{|\beta|=p-1}
    \left(\begin{array}{l}
    p \\
    \beta
    \end{array}\right)
    \theta^{\beta^2} u^{\eps}(x, t)^\beta \sum_{k=1}^m \theta_k^{2 \beta_k+1}f_k^{\eps}\left(u^{\eps}(x, t)\right)\, dx,
  \end{split}
\end{equation}
where we used \eqref{Q2}.

From \eqref{A4} and Lemma \ref{lem-f-sum},  we choose the components of $\theta = (\theta_i)$ inductively so that $\theta_i$ are sufficiently large that the previous positive definiteness condition of $\widetilde{B}_{k, l}(\beta)$ is satisfied, and
  \begin{equation}\label{star}
     \begin{split}
         \theta_i\geq g_i(\theta_{i+1}^{2p-1},\cdots,\theta_{m}^{2p-1}) \quad\text{for}\quad i=1,\cdots,m-1,
      \end{split}
  \end{equation}
where $g_i$ are functions constructed in Lemma \ref{lem-f-sum}. Note that $\theta_i \leq \theta_i^{2\beta_i + 1} \leq \theta_i^{2p - 1}$. Since $g_i$ is componentwise increasing,  the relation \eqref{star} implies
\begin{equation*}
	\theta_i^{2\beta_i + 1} \geq g_i\bra{\theta_{i+1}^{2\beta_{i+1}+1}, \ldots, \theta_m^{2\beta_m + 1}}, \quad \forall i=1,\ldots, m-1.
\end{equation*}
Now we can apply Lemma \ref{lem-f-sum}, to obtain some $K_{\widetilde{\theta}}$  so that for all $\beta\in\mathbb{Z}_+$ with $|\beta|=p-1$, we have
  \begin{equation*}
     \begin{split}
         \sumi \theta_{i}^{2\beta_i+1}f^{\eps}_i(x,t,u^{\eps})\leq K_{\widetilde{\theta}}\bra{1+\sumi (u^{\eps}_i)^r}, \quad \forall (x,t,u)\in \Omega\times\mathbb{R}_+\times \mathbb{R}^m_+.
      \end{split}
  \end{equation*}
It follows that there exists $C_p>0$, such that \eqref{eq-E} implies
\begin{equation*}
  \begin{split}
    \mathscr{L}_p^{\prime}(t)+&\alpha_p \sum_{k=1}^m \int_{\Omega}u_k^{\eps}(x,t)^{b+p-2} \left|\nabla u^{\eps}_k(x, t)\right|^2 \,dx \leq C\left(\intO\sumi u_i^{\eps}(x,t)^{p-1+r}\,dx+1\right).
  \end{split}
\end{equation*}
Thus, we have
\begin{equation}\label{p}
  \begin{split}
    \mathscr{L}_p^{\prime}(t)+&\alpha_p \sum_{k=1}^m \int_{\Omega}\left(u_k^{\eps}(x,t)^{b+p-2} \left|\nabla u^{\eps}_k(x, t)\right|^2 +u_i^{\eps}(x,t)^{b+p}\right)\,dx\\
    & \leq C\left(\intO\sumi \left(u_i^{\eps}(x,t)^{p-1+r} + u_i^{\eps}(x,t)^{b+p}\right) \,dx+1\right).
  \end{split}
\end{equation}
Combining Lemma \ref{lem-func-ineq}, we get
\begin{equation}\label{eq-p}
  \begin{split}
    \mathscr{L}_p^{\prime}(t)+&\frac{\alpha_p}{2} \sum_{k=1}^m \int_{\Omega}\left(u_k^{\eps}(x,t)^{b+p-2} \left|\nabla u^{\eps}_k(x, t)\right|^2 +u_i^{\eps}(x,t)^{b+p}\right)\,dx \leq C(T,p).
  \end{split}
\end{equation}
This implies
\begin{equation*}
  \begin{split}
    \mathscr{L}_p^{\prime}(t) + \sigma\mathscr{L}_p(t)  \leq C(T,p).
  \end{split}
\end{equation*}
Thus, we have
\begin{equation*}
\sup _{t \in(0, T)}\|u_i^{\eps}(t)\|_{L^p(\Omega)} \leq C_{T, p} \quad \forall i=1, \ldots, m .
\end{equation*}
This finishes the proof of Lemma \ref{lem-L-p-bound}.
\end{proof}

\begin{lemma}\label{lem-L-infty-bound}
Assume \eqref{D:cond-1}, \eqref{D:cond-2},  \eqref{Q1}, \eqref{Q2}, \eqref{A1}, \eqref{A2},  \eqref{A3}, \eqref{A5} and \eqref{A4} with $1\leq r<1+b+\frac{2}{N}$. Then for any $T>0$, the solution of \eqref{appro-system} is bounded in $L^{\infty}$ in time, i.e.,
\begin{equation*}
  \|u_i^{\eps}(t)\|_{L^{\infty}(Q_T)} \leq C_{T} \quad \forall i=1, \cdots, m
\end{equation*}
for some constant $C_T$ depending on $T$ and independent of $\eps>0$.
\end{lemma}

\begin{proof}
  The proof is similar to the proof of Lemma 2.3 in Ref. \cite{fellner2020Global-reg}. For the convenience of reading, we give the specific proof process in the Appendix \ref{appendix}.
\end{proof}

\subsection{Passing to the limit - Global existence}\label{subsec-existence}

\begin{lemma}\label{lem-nabla-u}
  Assume \eqref{D:cond-1}. For any $k>0$, we have
  \begin{equation*}
    \int^T_0\intO \chi_{\{|u^{\eps}_i|\leq k\}} \Phi(u^{\eps}) |\nabla u^{\eps}_i|^2 \,dxdt \leq \frac{k}{\lambda} \left(\|u_0\|_{\LO 1}+\|f^{\eps}_i(u^{\eps})\|_{\LQ 1}\right).
  \end{equation*}
\end{lemma}

\begin{proof}
  Let $T_k(z)$ be defined
  \begin{equation*}
    T_k(z)=\begin{cases}
    -k  & if z\leq-k,\\
    z   & if -k<z<k,\\
    k   & if z\geq k.
    \end{cases}
  \end{equation*}
  Define $S_k(z)=\int^z_0 T_k(\tau) \,d\tau$ and multiply \eqref{appro-system} by $T_k(u^{\eps}_i)$ we obtain
  \begin{equation*}
    \begin{split}
      \intO S_k(u_i^{\eps})\,dx & + \int^T_0\intO D_i(x,t) \Phi(u^{\eps})  \nabla u_i^{\eps} \cdot \nabla(T_k(u^{\eps}_i))\, dxdt\\
      &=\intO S_k(u_0^{\eps})\,dx + \int^T_0\intO f^{\eps}_i(u^{\eps})T_k(u^{\eps}_i)\,dxdt.
    \end{split}
  \end{equation*}
  By applying the properties of $T_k$ and $S_k$, the right hand is bounded by
  \begin{equation*}
    \begin{split}
      \intO S_k(u_0^{\eps})\,dx + \int^T_0\intO f^{\eps}_i(u^{\eps})T_k(u^{\eps}_i)\,dxdt\leq k \left(\|u_0\|_{\LO1}+\|f_i^{\eps}\|_{\LQ1} \right).
    \end{split}
  \end{equation*}
  From \eqref{D:cond-1} and $ \nabla(T_k(u^{\eps}_i))=\chi_{\{|u^{\eps}_i|\leq k\}} \nabla u^{\eps}_i$, it follows that
  \begin{equation*}
    \begin{split}
       \int^T_0\intO D_i(x,t) \Phi(u^{\eps}) \nabla u_i^{\eps} \cdot \nabla(T_k(u^{\eps}_i))\, dxdt &\geq \lambda \int^T_0\intO \chi_{\{|u^{\eps}_i|\leq k\}} \Phi(u^{\eps})  |\nabla u_i^{\eps}|^2 \, dxdt.
    \end{split}
  \end{equation*}
  Finally, by using $S_k(u^{\eps}_i)\geq 0$ we obtain our desired estimate.
\end{proof}

\begin{lemma}\label{lem-nabla-u-2}
  Assume \eqref{D:cond-1}. For any $\beta>0$, there exists constants $C$ such that
  \begin{equation}\label{eq-nabla-u}
    \int^T_0\intO \frac{\Phi(u^{\eps}) |\nabla u^{\eps}_i|^2}{(1+|u^{\eps}_i|)^{1+\beta} } \,dxdt \leq C \left(\|u_0\|_{\LO 1} + \|f^{\eps}_i(u^{\eps})\|_{\LQ 1}\right).
  \end{equation}
\end{lemma}

\begin{proof}
  Let $M:=\|u_0\|_{\LO 1}+\|f^{\eps}_i(u^{\eps})\|_{\LQ 1}$, we can apply Lemma \ref{lem-nabla-u} to obtain
  \begin{equation*}
    \begin{split}
      \int^T_0\int_{\Omega} \frac{ \Phi(u^{\eps})  |\nabla u^{\eps}_i|^2}{(1+|u^{\eps}_i|)^{1+\beta} } \,dxdt &=\sum^{\infty}_{j=0} \int^T_0\intO \mathbf{1}_{\{2^j-1\leq |u^{\eps}_i|\leq 2^{j+1}-1\}}  \frac{ \Phi(u^{\eps}) |\nabla u^{\eps}_i|^2}{(1+|u^{\eps}_i|)^{1+\beta} } \,dxdt\\
      &\leq \sum^{\infty}_{j=0}2^{-j(1+\beta)} \int^T_0\intO \mathbf{1}_{\{|u^{\eps}_i|\leq 2^{j+1}-1\}}  \Phi(u^{\eps})  |\nabla u^{\eps}_i|^2 \,dxdt\\
      &\leq \frac{M}{\lambda} \sum^{\infty}_{j=0}2^{-j(1+\beta)}  2^{j+1}\\
      &\leq \frac{2M}{\lambda} \sum^{\infty}_{j=0}(2^{-\beta})^{j}.
    \end{split}
  \end{equation*}
  Thus \eqref{eq-nabla-u} holds for $C:=\frac{2}{\lambda} \sum^{\infty}_{j=0}(2^{-\beta})^{j}$, which is finite because of $\beta>0$.
\end{proof}

\begin{proof}[Proof of Theorem \ref{th1}-Global existence]
From subsection \ref{subsec-uni-eps} we have the following bound
 \begin{equation*}
  \|u_i^{\eps}(t)\|_{L^{\infty}(Q_T)} \leq C_{T} \quad \forall i=1, \cdots, m.
\end{equation*}
Due to the polynomial growth $\eqref{A5}$, we can get
\begin{equation}\label{eq-f}
  \|f_i(u_i^{\eps}(t))\|_{L^{\infty}(Q_T)} \leq C_{T} \quad \forall i=1, \cdots, m.
\end{equation}
By multiply \eqref{appro-system} by $u_i^{\eps}$ then integrating on $Q_T$ gives
\begin{equation*}
  \begin{split}
    \frac{1}{2}\|u^{\eps}_i(T)\|^2_{\LO 2}&+\int^T_0\intO D_i(x,t) \Phi(u^{\eps})  \nabla u_i^{\eps} \cdot \nabla u_i^{\eps} \,dxdt\\
    &=\frac{1}{2}\|u^{\eps}_i(0)\|^2_{\LO 2}+\int^T_0\intO f_i^{\eps}(u^{\eps}) u_i^{\eps} \,dxdt\\
    &\leq \frac{1}{2}\|u^{\eps}_i(0)\|^2_{\LO 2}+\|f_i^{\eps}(u^{\eps}) \|_{\LQ {\infty}}\|u_i^{\eps}\|_{\LQ 1}\\
    &\leq C\left(T, |\Omega|, \|u_{i,0}\|_{\LO{\infty}}\right).
  \end{split}
\end{equation*}
By \eqref{D:cond-1}, \eqref{Q2}  and H\"{o}lder inequality give
\begin{equation*}
  \begin{split}
    \int^T_0&\intO D_i(x,t) \Phi(u^{\eps})  \nabla u_i^{\eps} \cdot \nabla u_i^{\eps} \,dxdt \\
    &\geq \lambda M_i \int^T_0\intO  (u_i^{\eps})^b |\nabla u_i^{\eps}|^2\,dxdt \\
    &=\lambda M_i (\frac{2}{b+2})^2 \int^T_0\intO  |\nabla ((u_i^{\eps})^{\frac{b+2}{2}})|^2\,dxdt\\
    &=\lambda  M_i (\frac{2}{b+2})^2 \|\nabla (( u_i^{\eps})^{\frac{b+2}{2}})\|_{\LQ 2}^2.
  \end{split}
\end{equation*}
In particular
\begin{equation}\label{eq-p(1)-es}
  \begin{split}
    \|\nabla ((u_i^{\eps})^{\frac{b+2}{2}})\|_{\LQ 2}\leq C\left(T, |\Omega|, \|u_{i,0}\|_{\LO{\infty}}\right).
  \end{split}
\end{equation}

By testing the equation \eqref{appro-system} with $\varphi\in L^2(0,T; H^1(\Omega))$, 
we have
\begin{equation*}
\begin{split}
  &\int^T_0\intO\pa_tu_i^{\eps} \varphi\,dxdt=-\int^T_0\intO D_i(x,t) \Phi(u^{\eps}) \nabla u_i^{\eps} \cdot\nabla \varphi \,dxdt+\int^T_0\intO f_i^{\eps}(u^{\eps}) \varphi \,dxdt\\
  &=-\int^T_0\intO D_i(x,t) \Phi(u^{\eps})^{\frac{1}{2}} (1+u_i^{\eps})^{\frac{1+\beta}{2}} \frac{ \Phi(u^{\eps})^{\frac{1}{2}}\nabla u_i^{\eps}}{(1+u_i^{\eps})^{\frac{1+\beta}{2}}} \cdot\nabla \varphi \,dxdt +\int^T_0\intO f_i^{\eps}(u^{\eps}) \varphi \,dxdt\\
  &\leq \|D_i(x,t)\|_{\LQ{\infty}} \|\Phi(u^{\eps})\|^{\frac{1}{2}}_{\LQ{\infty}} (1+\|u_i^{\eps}\|_{\LQ{\infty}})^{\frac{1+\beta}{2}} \|\frac{\Phi(u^{\eps})^{\frac{1}{2}}|\nabla u_i^{\eps}|}{(1+u_i^{\eps})^{\frac{1+\beta}{2}}}\|_{\LQ2} \|\nabla\varphi\|_{\LQ2}\\
  & \quad +\|f_i^{\eps}(u^{\eps})\|_{\LQ{\infty}} (|\Omega| T)^{\frac 1 2}\|\varphi\|_{\LQ2}\\
  &\leq C\left(T, |\Omega|,\|u_{i,0}\|_{\LO{\infty}}, \|D_i(x,t)\|_{\LQ{\infty}}\right) \|\varphi\|_{L^2(0,T; H^1(\Omega))},
\end{split}
\end{equation*}
where we use \eqref{D:cond-2}, \eqref{Q3}, \eqref{eq-nabla-u} and \eqref{eq-f}.
Thus we can get
\begin{equation}\label{eq-pa-t}
\|\pa_t u_i^{\eps}\|_{L^2(0,T; (H^1(\Omega))')}\leq C\left(T, |\Omega|, \|u_{i,0}\|_{\LO{\infty}}\right).
\end{equation}

From Lemma \ref{lem-L-infty-bound}, \eqref{eq-p(1)-es} and \eqref{eq-pa-t}, we can  apply a nonlinear version of the well known Aubin-Lions Lemma (see e.g. \cite[Theorem 1.1]{Moussa2016Some}) to ensure
\begin{equation}\label{eq-strong-conv}
  u_i^{\eps}\stackrel{\eps\rightarrow 0}{\longrightarrow}u_i \quad \text{strongly in } \LQ2
\end{equation}
Thanks to the $L^{\infty}$ bound of Lemma \ref{lem-L-infty-bound}, this convergence in fact holds in $\LQ p$ for any $1\leq p<\infty$.

Since the sequence $(u_i^\eps)^{\frac{b+2}{2}}$ is uniformly bounded in $L^2(0,T; W^{1,2}(\Omega))$, there exists a subsequence (still denoted by $(u_i^\eps)$) and $v\in \LQ 2$ such that 
\begin{align*}
	\nabla (u_i^\eps)^{\frac{b+2}{2}} \rightharpoonup v \quad \text{weekly in } \LQ 2.
\end{align*}
Now, we are ready to identify the weak limit $v$ by the following computation: for testing functions $\varphi \in C^{\infty}_0(Q_T)$ we have 
\begin{align*}
	\intQT v \cdot \varphi \,dxdt &= \lim_{\eps\rightarrow 0} \intQT \nabla (u_i^\eps)^{\frac{b+2}{2}} \cdot \varphi \,dxdt\\
	& = -  \lim_{\eps\rightarrow 0} \intQT  (u_i^\eps)^{\frac{b+2}{2}} \cdot \nabla \varphi \,dxdt\\
	& = -  \intQT  u_i^{\frac{b+2}{2}} \cdot \nabla \varphi \,dxdt.
\end{align*}
This implies that the function $u_i^{\frac{b+2}{2}}$ is weakly differentiable with weak derivative $v$, i.e., $\nabla u_i^{\frac{b+2}{2}} = v$.
Thus
\begin{align}\label{eq-weak-convergence-nabla}
	\nabla (u_i^\eps)^{\frac{b+2}{2}} \rightharpoonup \nabla u_i^{\frac{b+2}{2}} \quad \text{weekly in } \LQ 2.
\end{align}

\medskip
It remains to pass  to the limit $\eps\rightarrow 0$ in the weak formulation of \eqref{appro-system} with $\eta\in L^2(0,T; H^1(\Omega))$,
\begin{equation}\label{eq-ori-limit}
\begin{split}
\int^T_0\langle \pa_t u_i^{\eps}, \eta\rangle_{(H^1(\Omega))', H^1(\Omega)} \,dt &+\int^T_0\intO D_i(x,t) \Phi(u^{\eps})  \nabla u_i^{\eps}\cdot \nabla \eta \,dxdt\\
&=\int^T_0\intO f_i^{\eps}(u^{\eps}) \eta \,dxdt.
\end{split}
\end{equation}
The convergence of the first term on the left hand side and the term on the right hand side of \eqref{eq-ori-limit} is immediate.

The convergence
\begin{equation}\label{eq-conv-term2}
\begin{split}
\int^T_0\intO D_i(x,t)  \Phi(u^{\eps})  \nabla u_i^{\eps}\cdot \nabla \eta \,dxdt \stackrel{\eps\rightarrow 0}{\longrightarrow} \int^T_0\intO D_i(x,t) \Phi(u) \nabla u_i\cdot \nabla \eta \,dxdt \quad \text{ in } Q_T
\end{split}
\end{equation}
follows from $\frac{\Phi(u^{\eps})}{(u^{\eps}_i)^{\frac{b}{2}}}   \stackrel{\eps\rightarrow 0}{\longrightarrow} \frac{\Phi(u)}{u_i^{\frac{b}{2}} } $ a.e. in $Q_T$, Lemma \ref{lem-L-infty-bound} and the equation \eqref{eq-weak-convergence-nabla}.

Passing to the limit in the weak formulation \eqref{eq-ori-limit}, we have
\begin{equation}\label{eq-limit}
\begin{split}
\int^T_0 \intO \pa_t u_i \eta \,dxdt  &+\int^T_0\intO D_i(x,t) \Phi(u) \nabla u_i\cdot \nabla \eta \,dxdt=\int^T_0\intO f_i(u) \eta \,dxdt.
\end{split}
\end{equation}
We obtain that $u=\left(u_1,\cdots,u_m\right)$ is a global weak solution to \eqref{system} and additionally
\begin{equation*}
  \|u_i\|_{\LQ{\infty}}\leq C_T, \quad \forall i=1,\cdots,m.
\end{equation*}
\end{proof}

\section{Uniform-in-time boundness}\label{sub-uni}

\begin{lemma}\label{L1-uni-bound}
Assume \eqref{A1}, \eqref{A2} and \eqref{A3} with either $K_1<0$ or $K_1=K_2=0$. Then, there exists a constant M independent of time such that
\begin{equation}\label{eq-uni-l1}
  \sup_{t\geq 0}\|u_i(t)\|_{\LO1}\leq M:=\|u_{i0}\|_{\LO1}, \quad \forall i=1,\cdots,m.
\end{equation}
\end{lemma}

\begin{proof}
  Similar to the proof of Lemma \ref{L1-bound}, we have
  \begin{equation*}
    \frac{d}{dt}\sumi\intO c_iu_i \,dx \leq \intO \left(K_1 \sumi u_i^{\eps}(x,t) + K_2\right) \,dx.
  \end{equation*}
  Integrating over $(s,t)$ we have
   \begin{equation*}
    \sumi c_i\intO u_i(x,t) \,dx \leq \sumi c_i\intO u_i(x,s) \,dx +\int^t_s\intO K_1 \sumi u_i(x,r) \,dx dr+ K_2 |\Omega|(t-s)
  \end{equation*}
  for all $t>s\geq0$.

  If $K_1=K_2=0$, the bound \eqref{eq-uni-l1} follows immediately.

  If $K_1<0$, we get for some constant $\sigma>0$ and for all $t>s\geq0$,
  \begin{equation}\label{eq-mod}
  \begin{split}
    \sumi c_i\intO u_i(x,t) \,dx &+ \sigma \int^t_s \left(\sumi c_i\intO  u_i(x,r) \,dx\right) dr \\
    &\leq \sumi c_i\intO u_i(x,s) \,dx+ K_2 |\Omega|(t-s).
  \end{split}
  \end{equation}
  Define
  \begin{equation*}
    \psi(t)= \sumi c_i\intO u_i(x,t) \,dx ~\text{ and } ~\phi(s) =  \int^t_s \left(\sumi c_i\intO  u_i(x,r) \,dx\right) dr.
  \end{equation*}
  It follows from \eqref{eq-mod} that
\begin{equation*}
  \phi^{\prime}(s)=-\sum_{i=1}^m c_i \int_{\Omega} u_i(\cdot, s) d x \leq-\psi(t)-\sigma \phi(s)+K_2|\Omega|(t-s),
\end{equation*}
which leads to
\begin{equation*}
\left(e^{\sigma s} \phi(s)\right)^{\prime}+e^{\sigma s} \psi(t) \leq K_2|\Omega| e^{\sigma s}(t-s) .
\end{equation*}

Integrating with respect to $s$ on $(0, t)$, and using $\phi(t)=0$, we have
\begin{equation*}
-\phi(0)+\psi(t) \frac{e^{\sigma t}-1}{\sigma} \leq \frac{K_2|\Omega|}{\sigma}\left(-t+\frac{e^{\sigma t}-1}{\sigma}\right) .
\end{equation*}
Since $\sigma \phi(0)-K_2|\Omega| t \leq \sum_{i=1}^m c_i \int_{\Omega}u_{i 0}(x)\,d x$ (see \eqref{eq-mod}) it follows that
\begin{equation*}
\psi(t) \leq\left(e^{\sigma t}-1\right)^{-1} \sum_{i=1}^m c_i \int_{\Omega} u_{i, 0}(x) d x+K_2|\Omega| \sigma^{-1},
\end{equation*}
which finishes the proof of  Lemma \ref{L1-uni-bound}.
\end{proof}

\begin{proof}[Proof of Theorem \ref{th1}-uniform-in-time boundness]
We first show that for any $1\leq p<\infty$, there exists a constant $C_p>0$ such that
\begin{equation}\label{eq-lp-uni}
    \sup_{t\geq 0}\|u_i(t)\|_{\LO p}\leq C_p, \quad \forall i=1,\cdots,m.
\end{equation}
Indeed, by using the $L^p$-energy function $\mathscr{L}_p[u]$ defined in \eqref{def-Lp-energy-fun} and similar to Lemma \ref{lem-L-p-bound}, we obtain
\begin{equation*}
  \begin{split}
    \mathscr{L}_p^{\prime}(t)+&\alpha_p \sum_{k=1}^m \int_{\Omega}u_k(x,t)^{b+p-2} \left|\nabla u_k(x, t)\right|^2 \,dx \leq C\left(\intO\sumi u_i(x,t)^{p-1+r}\,dx+1\right).
  \end{split}
\end{equation*}
Combining Lemmas \ref{lem-func-ineq} and \ref{L1-uni-bound}, we get
\begin{equation*}
  \begin{split}
    \mathscr{L}_p^{\prime}(t)+&\frac{\alpha_p}{2} \sum_{k=1}^m \int_{\Omega}\left(u_k(x,t)^{b+p-2} \left|\nabla u_k(x, t)\right|^2 +u_k(x,t)^{b+p}\right)\,dx \leq C(p).
  \end{split}
\end{equation*}
This implies
\begin{equation*}
  \begin{split}
    \mathscr{L}_p^{\prime}(t) + \sigma\mathscr{L}_p(t) \leq C(p).
  \end{split}
\end{equation*}
Thus, we have
\begin{equation*}
\sup _{t\geq 0}\|u_i(t)\|_{L^p(\Omega)} \leq C_{p} \quad \forall i=1, \ldots, m.
\end{equation*}

Next we will show that the solution is bounded uniformly in time in sup norm. We use a smooth time-truncated function $\psi: \R\rightarrow [0,1]$ with
\begin{equation*}
  \psi(s)=
  \begin{cases}
    0, &s\leq0\\
    1, &s\geq 1
  \end{cases}
\end{equation*}
and $0\leq \phi'\leq C$, and its shifted version $\psi_{\tau}(\cdot)=\psi(\cdot,-\tau)$ for any $\tau\in\mathbb{N}$.  Let $\tau\in\mathbb{N}$ be arbitrary. It is straightforward to show that since $u=(u_i)_{i=1,\cdots,m}$ is a weak solution to \eqref{system}, the function $\psi_{\tau}u=(\psi_{\tau}u_i)_{i=1,\cdots,m}$ is a weak solution to the following
\begin{equation*}
  \begin{cases}
  \pa_t (\psi_{\tau}u_i)-\nabla\cdot(D_i(x,t) \Phi(\psi_{\tau}u) \nabla (\psi_{\tau}u_i))=\psi'_{\tau}u_i+\psi_{\tau}f_i(u),  &x\in \Omega, t\in(\tau,\tau+2)\\
  (D_i(x,t) \Phi(\psi_{\tau}u) \nabla (\psi_{\tau}u_i))\cdot \nu=0,   &x\in \pa \Omega, t\in(\tau,\tau+2)\\
  (\psi_{\tau}u_i)(x,\tau)=0,   &x\in \Omega.
  \end{cases}
\end{equation*}

Thanks to \eqref{eq-lp-uni} and the polynomial growth \eqref{A5}, we have
\begin{equation*}
  \psi_{\tau}' u_i+\psi_{\tau}f_i(u)\leq G_i(x,t,u):=C(1+\sum^m_{k=1} u_k^l)\in L^p(\Omega\times(\tau,\tau+2)),
\end{equation*}
where $1\leq p<\infty$.
Therefore, similar to the proof of Lemma \ref{lem-L-infty-bound}, we get
\begin{equation*}
  \|\psi_{\tau}u_i\|_{L^{\infty}(\Omega\times(\tau,\tau+2)}\leq C, \quad \forall i=1,\cdots,m,
\end{equation*}
where C is a constant independent of $\tau\in \mathbb{N}$. Thanks to $\psi_{\tau}\geq 0$ and $\psi|_{(\tau+1,\tau+2)}\equiv1$, we obtain finally the uniform-in-time bound
\begin{equation*}
\sup_{\tau\in \mathbb{N}}\|u_i\|_{L^\infty(\Omega\times(\tau,\tau+1))}\leq C,\quad \forall i=1,\cdots,m,
\end{equation*}
and the proof of Theorem \ref{th1} is complete.
\end{proof}

\section{Proof of Theorems \ref{th2}--\ref{th5}}\label{sub-proof}

\begin{proof}[Proof of Theorem \ref{th2}]
For the global existence, we only need to show that for any $1\leq p<\infty$ and any $T>0$, there exists $C_{T,p}>0$ such that
\begin{equation}\label{p-1}
\sup _{t \in(0, T)}\|u_i^{\eps}(t)\|_{L^p(\Omega)} \leq C_{T,p}, \quad \forall i=1,\cdots,m.
\end{equation}
The rest follows exactly as in Lemma \ref{lem-L-infty-bound}. To show \eqref{p-1}, we use the $L^p$-energy functions $\mathscr{L}_p(t)$ constructed in Lemma \ref{lem-L-p-bound} until \eqref{p} we end up with
\begin{equation}\label{e-p-1}
  \begin{split}
    \mathscr{L}_p^{\prime}(t)+&\alpha_p \sum_{i=1}^m \int_{\Omega}\left(u_i^{\eps}(x,t)^{b+p-2} \left|\nabla u^{\eps}_i(x, t)\right|^2 +u_i^{\eps}(x,t)^{b+p}\right)\,dx\\
    & \leq C\left(\sumi\intO \left(u_i^{\eps}(x,t)^{p-1+r} + u_i^{\eps}(x,t)^{b+p}\right) \,dx+1\right).
  \end{split}
\end{equation}
Since \eqref{condi-u-1} and \eqref{condi-growth1}, we apply Lemma \ref{lem-func-ineq} to estimate
\begin{equation*}
  \begin{split}
    &\intO\left(u_i^{\eps}(x,t)^{p-1+r} + u_i^{\eps}(x,t)^{b+p}\right) \,dx\\
    &\leq  \frac{\alpha_p}{2}(\int_{\Omega} u_i^{\eps}(x,t)^{p-2+b}|\nabla u_i^{\eps}(x,t)|^{2}\,dx+\int_{\Omega} u_i^{\eps}(x,t)^{b+p}\,dx)+C_{T},
  \end{split}
\end{equation*}
where $C_{T}$ depends on $\mathscr{F}(T)$. Inserting this into \eqref{e-p-1} yields
\begin{equation*}
  \begin{split}
    \mathscr{L}_p^{\prime}(t)+\sigma \mathscr{L}_p(t) \leq C(T,p),
  \end{split}
\end{equation*}
which implies \eqref{p-1}.

\red{For the uniform-in-time bounds, we use an argument similar to the proof of Theorem \ref{th1}.}
\end{proof}

\begin{proof}[Proof of Theorem \ref{th3}]
For the global existence, we only need to show that for any $1\leq p<\infty$ and any $T>0$, there exists $C_{T,p}>0$ such that
\begin{equation}\label{p-2}
\sup _{t \in(0, T)}\|u_i^{\eps}(t)\|_{L^p(\Omega)} \leq C_{T,p}, \quad \forall i=1,\cdots,m.
\end{equation}
The rest follows exactly as in Lemma \ref{lem-L-infty-bound}. To show \eqref{p-2}, we use the $L^p$-energy functions $\mathscr{L}_p(t)$ constructed in Lemma \ref{lem-L-p-bound} until \eqref{p}  we obtain
\begin{equation}\label{e-p-2}
  \begin{split}
    \mathscr{L}_p^{\prime}(t)+&\frac{\alpha_p}{2} \sum_{i=1}^m \int_{\Omega}\left(u_i^{\eps}(x,t)^{b+p-2} \left|\nabla u^{\eps}_i(x, t)\right|^2 +u_i^{\eps}(x,t)^{b+p}\right)\,dx\\
    & \leq C\left(\sumi\intO u_i^{\eps}(x,t)^{p-1+r} \,dx+1\right).
  \end{split}
\end{equation}
We integrate \eqref{e-p-2} in time to obtain
\begin{equation}\label{eq-t3-1}
  \begin{split}
    \sup_{t\in(0,T)}\mathscr{L}_p(t)+&\frac{\alpha_p}{2} \sum_{i=1}^m \int^T_0\int_{\Omega}\left(u_i^{\eps}(x,t)^{b+p-2} \left|\nabla u^{\eps}_i(x, t)\right|^2 +u_i^{\eps}(x,t)^{b+p}\right)\,dxdt\\
    & \leq C\sumi\int^T_0\intO u_i^{\eps}(x,t)^{p-1+r}  \,dxdt+C_{p,T}+\mathscr{L}_p(0).
  \end{split}
\end{equation}
 Denote by $y_i:=u_i^{\eps}(x,t)^{\frac{b+p}{2}}$. The left-hand side (LHS) of \eqref{eq-t3-1} can be estimated below by
\begin{equation}
  \begin{split}
    C\sumi \left( \|y_i\|^{\frac{2p}{b+p}}_{L^{\infty}(0,T; L^{\frac{2p}{b+p}}(\Omega))}+\|y_i\|^2_{L^2(0,T; H^1(\Omega))} \right)\leq \text{(LHS) of } \eqref{eq-t3-1}.
  \end{split}\label{eq-t3-2}
\end{equation}
For the right-hand side of \eqref{eq-t3-1}, we first consider
\begin{equation}\label{eq-t3-3}
  \begin{split}
   \vt>p-1+r
  \end{split}
\end{equation}
as a constant to be determined later. Of course we are only interested in the case when $p-1+r>q$, otherwise the right hand side of \eqref{eq-t3-1} is bounded thanks to \eqref{condi-u-2}. By the interpolation inequality we have
	\begin{equation}\label{eq-t3-4}
    \begin{split}
		\intQT u^{\eps}_i(x,t)^{p-1+r}dxdt &= \|u^{\eps}_i(x,t)\|_{\LQ{p-1+r}}^{p-1+r}\\
        &\leq \|u^{\eps}_i(x,t)\|_{\LQ{q}}^{\theta(p-1+r)}\|u^{\eps}_i(x,t)\|_{\LQ{\vt}}^{(1-\theta)(p-1+r)},
     \end{split}
	\end{equation}
	where $\theta\in (0,1)$ satisfies
	\begin{equation*}
		\frac{1}{p-1+r} = \frac{\theta}{q} + \frac{1-\theta}{\vt},
	\end{equation*}
	which implies
	\begin{equation*}
		(1-\theta)(p-1+r) = \frac{\vt(p-1+r-q)}{\vt - q}.
	\end{equation*}
	Using this, and taking into account \eqref{condi-u-2}, \eqref{eq-t3-4} implies
	\begin{equation}\label{eq-t3-5}
	\begin{aligned}
		\intQT u^{\eps}_i(x,t)^{p-1+r}dxdt &\leq \mathscr{F}(T)^{\theta(p-1+r)}\|u^{\eps}_i(x,t)\|_{\LQ{\vt}}^{\frac{\vt(p-1+r-q)}{\vt - q}}\\
		&= \mathscr{F}(T)^{\theta(p-1+r)}\bra{\intQT u^{\eps}_i(x,t)^{\vt}dxdt}^{\frac{p-1+r-q}{\vt -q}}\\
		&= C_T\bra{\intQT y_i^{\frac{2\vt}{b+p}}dxdt}^{\frac{p-1+r-q}{\vt - q}}.
	\end{aligned}
	\end{equation}
	For $p$ large enough, we choose $\vt$ close enough to (but bigger than) $p-1+r$ so that $H^1(\Omega)\hookrightarrow \LO{\frac{2\vt}{b+p}}$. That means $\vt$ is arbitrary for $N\leq 2$ and
	\begin{equation}\label{eq-t3-6}
		\frac{2\vt}{b+p} \leq \frac{2N}{N-2} \Leftrightarrow \vt \leq \frac{(b+p) N}{N-2} \quad\text{ for }\quad N\geq 3.
	\end{equation}
	Thus, we can use the Gagliardo-Nirenberg's inequality to estimate
	\begin{equation}\label{eq-t3-7}
		\intO y_i^{\frac{2\vt}{b+p}}dx = \|y_i\|_{\LO{\frac{2\vt}{b+p}}}^{\frac{2\vt}{b+p}} \leq C\|y_i\|_{H^1(\Omega)}^{\alpha\cdot \frac{2\vt}{b+p}}\|y_i\|_{\LO{\frac{2p}{b+p}}}^{(1-\alpha)\cdot \frac{2\vt}{b+p}}
	\end{equation}
	where $\alpha\in (0,1)$ satisfies
	\begin{equation*}
		\frac{b+p}{2\vt} = \bra{\frac 12 - \frac 1N}\alpha + \frac{(1-\alpha)(b+p)}{2p}.
	\end{equation*}
	From this
	\begin{equation*}
		\alpha\cdot\frac{2\vt}{b+p} = \frac{2N(\vt-p)}{Nb+2p} \quad \text{ and }
     \quad (1-\alpha)\cdot\frac{2\vt}{b+p} = \frac{2Np(b+p) - 2(N-2)p\vt}{(b+p)(Nb+2p)}.
	\end{equation*}
	Therefore, we obtain from \eqref{eq-t3-7} that
	\begin{equation*}
		\intO y_i^{\frac{2\vt}{b+p}}dx \leq C\|y_i\|_{H^1(\Omega)}^{\frac{2N(\vt-p)}{Nb+2p}}
        \|y_i\|_{\LO{\frac{2p}{b+p}}}^\frac{2Np(b+p) - 2(N-2)p\vt}{(b+p)(Nb+2p)}.
	\end{equation*}
	It follows that
	\begin{align*}
		\intQT y_i^{\frac{2\vt}{b+p}}dxdt &\leq C\int_0^T\|y_i\|_{H^1(\Omega)}^{\frac{2N(\vt-p)}{Nb+2p}}
       \|y_i\|_{\LO{\frac{2p}{b+p}}}^{\frac{2Np(b+p) - 2(N-2)p\vt}{(b+p)(Nb+2p)}} \,dt\\
		&\leq C\|y_i\|_{L^{\infty}(0,T;\LO{\frac{2p}{b+p}})}^{\frac{2Np(b+p) - 2(N-2)p\vt}{(b+p)(Nb+2p)}} \int_0^T\|y_i\|_{H^1(\Omega)}^{\frac{2N(\vt-p)}{Nb+2p}} \,dt.
	\end{align*}
	We choose
	\begin{equation}\label{eq-t3-8}
	\frac{2N(\vt-p)}{Nb+2p} \leq 2 \Leftrightarrow \vt \leq \frac{Np+Nb+2p}{N}=b+p+\frac{2p}{N},
	\end{equation}
	which is possible since $p - 1 + r < p + b+ \frac{2p}{N}$ for $p$ large enough. Thus, by H\"older's inequality,
	\begin{equation*}
		\intQT y_i^{\frac{2\vt}{b+p}}dxdt \leq C_T\|y_i\|_{L^{\infty}(0,T;\LO{\frac{2p}{b+p}})}^{\frac{2Np(b+p) - 2(N-2)p\vt}{(b+p)(Nb+2p)}}
          \|y_i\|_{L^2(0,T;H^1(\Omega))}^{\frac{2N(\vt-p)}{Nb+2p}}.
	\end{equation*}
	Inserting this into \eqref{eq-t3-5} yields
	\begin{equation}\label{eq-t3-9}
	\begin{split}
	\intQT &u_i^{\eps}(x,t)^{p-1+r}dxdt \leq C_T\bra{\|y_i\|_{L^{\infty}(0,T;\LO{\frac{2p}{b+p}})}^{\frac{2Np(b+p) - 2(N-2)p\vt}{(b+p)(Nb+2p)}}
          \|y_i\|_{L^2(0,T;H^1(\Omega))}^{\frac{2N(\vt-p)}{Nb+2p}}}^{\frac{p-1+r-q}{\vt - q}}
	\end{split}
    \end{equation}
where $\theta\in(0,1).$

Using Young's inequality of the form
	\begin{equation*}
		x^{\lambda_1}y^{\lambda_2} \leq \eps(x^{\frac{2p}{b+p}}+y^2) + C_{\eps} \quad \text{ for } \quad \lambda_1\frac{b+p}{2p}+\frac{\lambda_2}{2} < 1,
	\end{equation*}
	we can choose $\vt$ such that
    \begin{equation}\label{eq-t3-11}
     \begin{split}
       \frac{N(b+p) - (N-2)\vt}{(Nb+2p)}\frac{p-1+r-q}{\vt - q} +
       \frac{N(p-1+r-q)(\vt - p)}{(Nb+2p)(\vt - q)}<1.
     \end{split}
	\end{equation}
This is equivalent to
\begin{align}\label{eq-t3-12}
		 \vt > \frac{Nb(p-1+r)+2pq}{Nb+2(1+q-r)}.
\end{align}
	We now check that we can choose $\vt$ which satisfies all the conditions \eqref{eq-t3-3}, \eqref{eq-t3-6}, \eqref{eq-t3-8} and \eqref{eq-t3-12}. This is fulfilled provided
\begin{equation*}
		\frac{Nb(p-1+r)+2pq}{Nb+2(1+q-r)}<b+p+\frac{2p}{N}, \quad\text{which is equivalent to}\quad r < 1 + \frac{N}{N+2}b +  \frac{2q}{N+2},
	\end{equation*}
	and this can be obtain from our assumption \eqref{condi-growth2}.
	We estimate \eqref{eq-t3-9} further as
	\begin{equation*}
		\intQT u_i^{\eps}(x,t)^{p-1+r}dxdt \leq \eps\bra{\|y_i\|_{L^{\infty}(0,T;\LO{\frac{2p}{b+p}})}^{\frac{2p}{b+p}} + \|y_i\|_{L^2(0,T;H^1(\Omega))}^2} + C_{T,\eps}.
	\end{equation*}
	Thus
	\begin{equation*}
		\text{RHS of (\ref{eq-t3-1})} \leq \mathscr{L}_p(0) + \eps\sumi \bra{\|y_i\|_{L^{\infty}(0,T;\LO{\frac{2p}{b+p}})}^{\frac{2p}{b+p}} + \|y_i\|^2_{L^2(0,T;H^1(\Omega))}} + C_{T,\eps}.
	\end{equation*}
	Combining this with \eqref{eq-t3-2} gives us the desired estimate \eqref{p-2}.

\red{For the uniform-in-time bounds, We use an argument similar to the proof of Theorem \ref{th1}.}
\end{proof}

\begin{proof}[Proof of Theorem \ref{th4}] We give a formal proof as its rigor can be easily obtained through approximation. Since $h_k(\cdot)$ is convex, we have
	\begin{align*}
		\nabla\cdot(D_i\Phi(u)\nabla(h_i(u_i))) &= \nabla\cdot(D_i\Phi(u)h_i^{\prime}(u_i)\nabla u_i)\\
		&\geq h'(u_i)\nabla\cdot(D_i\Phi(u)\nabla u_i) + \lambda h_i''(u_i)\Phi(u)|\nabla u_i|^2\\
		&\geq h'(u_i)\nabla \cdot(D_i\Phi(u)\nabla u_i).
	\end{align*}
	Therefore, by defining $v_i:= h_i(u_i)\geq 0$ we have
	\begin{equation}\label{ee2}
		\pa_t v_i - \nabla\cdot(D_i\Phi(u)\nabla v_i) \leq G_i(x,t,u):= h_i^{\prime}(u_i)f_i(x,t,u), \quad x\in\Omega, \; t>0,
	\end{equation}
	with initial data $v_i(x,0) = h_i(u_{i,0}(x))$ and homogeneous Dirichlet boundary condition $v_i(x,t) = 0$ for $x\in\pa\Omega$ and $t>0$. Thanks to \eqref{H1} and  we have
	\begin{equation*}
		\sumi G_i(x,t,u) \leq K_5\sumi v_i + K_6
	\end{equation*}
	and
	\begin{equation*}
		A\begin{pmatrix}G_i(x,t,u)\\ \cdots\\ G_m(x,t,u)\end{pmatrix} \leq K_7\overrightarrow{1}\bra{\sumi v_i + 1}^r.
	\end{equation*}
	We can now reapply the methods in the proof of Theorem \ref{th1} to obtain $v_i \in L^\infty_{\text{loc}}(0,\infty;\LO{\infty})$, and in case $K_5<0$ or $K_5 = K_6 = 0$ in \eqref{H1},
	\begin{equation*}
		\text{ess}\sup_{t\geq 0}\|v_i(t)\|_{\LO{\infty}} <+\infty, \quad \forall i=1,\ldots, m.
	\end{equation*}
	Due to the assumption \eqref{H1}, the global existence and boundedness of \eqref{system} immediately hold.
\end{proof}

\begin{definition}\label{def_diff_bc}
	A vector of non-negative concentrations $u = (u_1, \ldots, u_m)$ is called a weak solution to \eqref{Sys_other_bc} on $(0,T)$ if
	\begin{equation*}
		u_i\in C([0,T]; \LO{2}), ~~ \Phi(u)\nabla u_i\in L^2(0,T;L^2(\Omega)), \quad f_i(u)\in L^2(0,T;\LO{2}),
	\end{equation*}
	with $u_i(\cdot,0)= u_{i,0}(\cdot)$ for all $i=1,\ldots, m$, and for any test function $\varphi \in L^2(0,T;H^1(\Omega))$ with $\pa_t\varphi \in L^2(0,T;H^{-1}(\Omega))$, it holds that
	\begin{equation}\label{weak_bc}
	\begin{aligned}
		\intO u_i(x,t)\varphi(x,t)dx\bigg|_{t=0}^{t=T} - \intQT u_i\pa_t\varphi  \,dxdt + \intQT D_i(x,t)\Phi(u)\nabla u_i \cdot \nabla \varphi \,dxdt\\
		+\alpha_i\int_0^T\int_{\pa\Omega}u_i\varphi d\mathscr{H}^{n-1} \,dt = \intQT f_i(x,t,u)\varphi \,dxdt.
	\end{aligned}
	\end{equation}
\end{definition}

\begin{proof}[Proof of Theorem \ref{th5}]
	The proof of this theorem is similar to that of Theorems \ref{th1}, \ref{th2} and \ref{th3}, except for the fact that the $L^1$-norm can be obtained in a different, and easier, way. Since $\varphi \equiv 1$ is an admissible test function we get from \eqref{weak_bc} and \eqref{A3} that
	\begin{align*}
		\sumi \intO c_iu_i(\cdot,t)  \,dx\bigg|_{t=0}^{t=T} + \sumi c_i\alpha_i\int_0^T\int_{\pa\Omega} u_id\mathscr{H}^{n-1} \,dt\\
		\leq K_1\intQT u_i  \,dxdt + K_2|\Omega|T.
	\end{align*}
	The $L^1$-bound is uniform in time in case $K_1<0$ or $K_1 = K_2 =0$ (using a similar idea to Lemma \ref{L1-uni-bound}).
\end{proof}

\section{Applications}\label{sec-appli}
We show the application of our results to SEIR(-D) model considered in \cite{auricchio2023well,viguerie2020diffusion,viguerie2021simulating} and its variants. First, the SEIR model introduced in  \cite{viguerie2020diffusion, viguerie2021simulating, auricchio2023well} reads as
\begin{equation}\label{sys-model}
	\begin{cases}
		\pa_t s= \alpha n - (1 - A_0/n) \beta_i si - (1 - A_0/n) \beta_e se - \mu s + \nabla\cdot(n \nu_s \nabla s), &x\in\Omega\\
		\pa_t e= (1 - A_0/n) \beta_i si + (1 - A_0/n) \beta_e se - \sigma e - \phi_e e - \mu e + \nabla\cdot(n \nu_e \nabla e), &x\in\Omega\\
		\pa_t i = \sigma e - \phi_d ni - \phi_r i - \mu i + \nabla\cdot(n \nu_i \nabla i), &x\in\Omega\\
		\pa_t r = \phi_r i + \phi_e e - \mu r + \nabla\cdot(n \nu_r \nabla r), &x\in\Omega,\\
         \nabla s \cdot \eta=\nabla e \cdot \eta=\nabla i \cdot \eta=\nabla r \cdot \eta=0, &x\in\pa\Omega,
	\end{cases}
\end{equation}
where the symbols $s$, $e$, $i$ and $r$ denote the susceptible-, exposed-, infected-, recovered populations, respectively, and $n:=s+e+i+r$ is the total living population, $\nu_s$, $\nu_e$, $\nu_i$, $\nu_r$, $\beta_i$, $\beta_e$, $\alpha$, $\mu$, $\phi$, $\sigma$, $\phi_d$, $\phi_r$ and $\phi_e$ are positive constants. The global existence of bounded weak solutions was shown in \cite{auricchio2023well} where it was imposed that all diffusion rates are the same, i.e. $\nu_s = \nu_e = \nu_i = \nu_r$, the term $1-A_0/n$ is replaced by a non-singular function $A(n)$, for instance $A(n) = (1-A_0/n)_+$, and the term $-\phi_dni$ is replaced by $-\phi_d i$. These are needed to show the total population $n$ is pointwise bounded from below by a positive constant, making the diffusion operators non-degenerate\footnote{In fact, thanks to the lower bound of $n$, it was possible to replace all diffusion operators by a common function $\kappa(n)$, where $\kappa: (0,\infty)\to (0,\infty)$ is a continuous function.}. The method therein seems to break down when the diffusion are different. By applying our main results in Theorem \ref{th1}, we show that these assumptions can be removed. We have the following result.
\begin{theorem}\label{th-Model}
	Assume all diffusion coefficients $\nu_s, \nu_e, \nu_i, \nu_r$, and other parameters in \eqref{sys-model} to be positive. Then, for any non-negative and bounded initial data $(s_0,e_0,i_0,r_0)\in L_+^{\infty}(\Omega)^4$, there is a global bounded weak solution to \eqref{sys-m}, i.e. for any $T>0$,
	\begin{equation}\label{model-ext}
		\sup_{t\in (0,T)}\left(\|s(t)\|_{\LO{\infty}}+\|e(t)\|_{\LO{\infty}}+ \|i(t)\|_{\LO{\infty}}+ \|r(t)\|_{\LO{\infty}} \right) <\infty.
	\end{equation}
	Moreover, if  $\alpha\leq \mu$, then this solution is bounded uniformly in time, i.e.
	\begin{equation}\label{model-uni}
		\sup_{t>0}\big(\|s(t)\|_{\LO{\infty}}+ \|e(t)\|_{\LO{\infty}}+ \|i(t)\|_{\LO{\infty}}+ \|r(t)\|_{\LO{\infty}} \big) < +\infty.
	\end{equation}
\end{theorem}

\begin{proof}
	
	It is noted that the nonlinearities in \eqref{sys-model} are not locally Lipschitz continuous around $0$. We circumvent this problem by considering the following approximating system for each $\delta>0$,
	\begin{equation}\label{sys-approximate-model}
		\begin{cases}
			\pa_t s^{\delta}= \alpha n^{\delta} - (1 - \frac{A_0}{n^{\delta} + \delta} ) \beta_i s^{\delta} i^{\delta} - (1 - \frac{A_0}{n^{\delta} + \delta} ) \beta_e s^{\delta} e^{\delta} - \mu s^{\delta} + \nabla\cdot(n^{\delta} \nu_s \nabla s^{\delta}), &x\in\Omega\\
			\pa_t e^{\delta}= (1 - \frac{A_0}{n^{\delta} + \delta} ) \beta_i s^{\delta} i^{\delta} + (1 - \frac{A_0}{n^{\delta} + \delta} ) \beta_e s^{\delta} e^{\delta} - \sigma e^{\delta} - \phi_e e^{\delta} - \mu e^{\delta} + \nabla\cdot(n^{\delta} \nu_e \nabla e^{\delta}), &x\in\Omega\\
			\pa_t i^{\delta} = \sigma e^{\delta} - \phi_d n^{\delta} i^{\delta} - \phi_r i^{\delta} - \mu i^{\delta} + \nabla\cdot(n^{\delta} \nu_i \nabla i^{\delta}), &x\in\Omega\\
			\pa_t r^{\delta} = \phi_r i^{\delta} + \phi_e e^{\delta} - \mu r^{\delta} + \nabla\cdot(n^{\delta} \nu_r \nabla r^{\delta}), &x\in\Omega,\\
			\nabla s^{\delta} \cdot\eta=\nabla e^{\delta} \cdot\eta=\nabla i^{\delta} \cdot\eta=\nabla r^{\delta} \cdot \eta=0, &x\in\pa\Omega,
		\end{cases}
	\end{equation}
	where $u^\delta = (s^\delta, e^\delta, i^\delta, r^\delta)$, $n^\delta = s^\delta + e^\delta + i^\delta + r^\delta$, $f^{\delta}_s, f^{\delta}_e, f^{\delta}_i$ and $f^{\delta}_r$ the nonlinearities in the equations of $s^{\delta}, e^{\delta}, i^{\delta}$ and $r^{\delta}$, respectively.
	Consequently
	\begin{equation*}
		\begin{cases}
			f^{\delta}_s(u^{\delta}) \leq \alpha n^{\delta}  + (A_0\beta_i + A_0\beta_e)s^{\delta} \\
			f^{\delta}_s(u^{\delta}) + f^{\delta}_e(u^{\delta})\leq \alpha  n^{\delta}\\
			f^{\delta}_s(u^{\delta}) + f^{\delta}_e(u^{\delta}) + f^{\delta}_i(u^{\delta}) \leq \alpha  n^{\delta}\\
			f^{\delta}_s(u^{\delta}) + f^{\delta}_e(u^{\delta}) + f^{\delta}_i(u^{\delta}) + f^{\delta}_r(u^{\delta}) \leq (\alpha - \mu) n^{\delta}.
		\end{cases}
	\end{equation*}
	Following from  Theorem \ref{th1}, we can obtain the global existence of a weak solution for the system \eqref{sys-approximate-model} and 
	\begin{align}\label{model-n-infty}
		\|s^{\delta}\|_{\LQ{\infty}}, \|e^{\delta}\|_{\LQ{\infty}}, \|i^{\delta}\|_{\LQ{\infty}}, \|r^{\delta}\|_{\LQ{\infty}}  \leq C_T,
	\end{align}
	where $C_T$ independent of $\delta$.
	Due to the polynomial growth of $f^{\delta}_s, f^{\delta}_e, f^{\delta}_i$ and $f^{\delta}_r$, we can get
	\begin{align*}
		\|f^{\delta}_s (u^{\delta})\|_{\LQ{\infty}}, \|f^{\delta}_e (u^{\delta})\|_{\LQ{\infty}}, \|f^{\delta}_i (u^{\delta})\|_{\LQ{\infty}}, \|f^{\delta}_r (u^{\delta})\|_{\LQ{\infty}}  \leq C_T.
	\end{align*}
	By multiply the first equation of \eqref{sys-approximate-model} by $s^{\delta} $ then integrating on $\Omega_T$ gives 
	\begin{align}\label{b1}
		\sup_{t\in(0,T)}\intO s^{\delta}(t)^2 + \nu_s \intQT n^{\delta}|\nabla s^{\delta}|^{2}dxdt \leq C_T.
	\end{align}
	Since $n^\delta \ge s^\delta$, it follows
	\begin{align}\label{model-nabla-s}
		C_T \ge \intQT s^\delta|\nabla s^\delta|^2dxdt \ge C \|\nabla  (s^{\delta})^{\frac 32} \|_{\LQ{2}}^2.
	\end{align}
	By testing the first equation of \eqref{sys-approximate-model} with $\phi \in L^2(0,T; H^1(\Omega))$, we have
	\begin{align*}
		&\intQT \pa_t s^{\delta} \phi \,dxdt = - \nu_s\intQT n^{\delta}  \nabla s^{\delta} \cdot\nabla \phi \,dxdt + \intQT f^{\delta}_s(u^{\delta}) \phi \,dxdt\\
		&\le \nu_s\|\sqrt{n^\delta}\|_{\LQ{\infty}}\| \sqrt{n^\delta}|\nabla s^\delta|\|_{\LQ{2}}\|\nabla \phi\|_{\LQ{2}} + \|f_s^{\delta}(u^\delta)\|_{\LQ{2}}\|\phi\|_{\LQ{2}}\\
		&\leq C_T \|\phi\|_{L^2(0,T; H^1(\Omega))},
	\end{align*}
	where we used \eqref{model-n-infty} and \eqref{b1}. Thus, 
	\begin{align*}
		\|\pa_t s^{\delta}\|_{L^2(0, T;(H^1(\Omega))')} \leq C_T.
	\end{align*}
	Applying a nonlinear version of Aubin-Lions Lemma, see e.g. \cite[Theorem 1]{Moussa2016Some} to ensure
	\begin{equation}\label{eq-strong-conv-s^delta}
		s^{\delta}\stackrel{\delta\rightarrow 0}{\longrightarrow} s \quad \text{strongly in } \LQ2
	\end{equation}
	Thanks to the $L^{\infty}$ bound of $s^{\delta}$, this convergence in fact holds in $\LQ p$ for any $1\leq p<\infty$. By the inequality \eqref{model-nabla-s},
	we obtain 
	\begin{align*}
		\nabla s^\delta \rightharpoonup \nabla s, \quad \text{weakly in } \LQ 2.
	\end{align*}
	Similarly, for  $e^{\delta}, i^{\delta}$ and $r^{\delta}$, we can get the strong convergence in $\LQ{p}$ for any $1\leq p<\infty$ and weak convergence in $L^2(0, T; H^1(\Omega))$. Now, let $\delta\rightarrow0$ in the weak formulation of \eqref{sys-approximate-model} with $\psi\in L^2(0, T; (H^1(\Omega))')$ and using \cite[Lemma A.2]{fischer2022global}, we get the weak solution $u= (s, e, i, r)$ of \eqref{sys-model}.
	
	\medskip
	
	Furthermore, if  $\alpha \leq \mu$, we have  the uniform-in-time boundedness first for solutions to \eqref{sys-approximate-model} and consequently for \eqref{sys-model}.
\end{proof}

\medskip
\noindent Thanks to the robustness of our approach, we can also consider a variant \eqref{sys-model} where
\begin{itemize}
	\item $\beta_i$ and $\beta_e$ are functions of $n$ or the combinations of $n$ and  species $i$ or $e$ (see the original model proposed in \cite{viguerie2020diffusion});
	\item the diffusion rates are functions of $(x,t)$, which shows the possibly high heterogeneity of the environment; and
	\item and the other rates $\alpha, \mu, \sigma, \phi_e, \phi_d, \phi_r$ are non-negative functions of $(x,t)$,
\end{itemize}
and therefore the system reads as
\begin{equation}\label{sys-model-2}
	\begin{cases}
		\pa_t s= \alpha n - (1-\frac{A_0}{n}) \beta_i(i,n) si - (1-\frac{A_0}{n}) \beta_e(e,n) se - \mu s + \nabla\cdot(n \nu_s(x,t) \nabla s),\\
		\pa_t e= (1-\frac{A_0}{n}) \beta_i(i,n) si + (1-\frac{A_0}{n}) \beta_e(e,n) se - \sigma e - \phi_e e - \mu e + \nabla\cdot(n \nu_e(x,t) \nabla e),\\
		\pa_t i = \sigma e - \phi_dn i - \phi_r i - \mu i + \nabla\cdot(n \nu_i (x,t)\nabla i),\\
		\pa_t r = \phi_r i + \phi_e e - \mu r + \nabla\cdot(n \nu_r(x,t) \nabla r), \\
           \nabla s \cdot\eta=\nabla e \cdot\eta=\nabla i \cdot\eta=\nabla r \cdot\eta=0.
	\end{cases}
\end{equation}
We assume the following:
\begin{enumerate}[label=(M\theenumi'),ref=M\theenumi']
	\item\label{M1'} $\beta_i(i,n), \beta_e(e,n)$ are non-negative continuous functions and satisfy
	\begin{equation*} 
		\beta_i(i,n)  \leq c_i \bra{1 + i + n}, \quad \beta_e(e,n) \le c_e\bra{1 + e+n},
	\end{equation*}
	for some constants $c_i, c_e>0$,
	\item\label{M2'} for $z\in \{s,e,i,r\}$, $\nu_z: \Omega\times[0,\infty) \to \R^{n\times n}$ such that $\nu_z\in L^\infty(\Omega\times(0,T);\R^{n\times n})$ for each $T>0$, and there is a constant $\lambda_z>0$ such that
	\begin{equation*}
		\lambda_z|\xi|^2 \le \xi^\top \nu_z(x,t) \xi, \quad \forall (x,t)\in\Omega\times[0,\infty), \quad \forall \xi \R^n;
	\end{equation*}
	\item \label{M3'} all non-negative functions $\alpha, \mu, \sigma, \phi_e, \phi_d, \phi_r: \Omega\times[0,\infty)$ are bounded by a common constant, i.e. $\exists M>0$ such that
	\begin{equation*}
		(\alpha + \mu + \sigma + \phi_e + \phi_d+\phi_r)(x,t) \le M, \quad \forall (x,t)\in \Omega\times [0,\infty).
	\end{equation*}
\end{enumerate}

The following theorem follows directly from Theorem \ref{th1}.
\begin{theorem}
	Assume \eqref{M1'}--\eqref{M3'}. Then, for any non-negative bounded initial data, there exists a  global bounded  weak solution to \eqref{sys-model-2}, i.e. for any $T>0$,
	\begin{equation*}
		\sup_{t\in (0,T)}\big(\|s(t)\|_{\LO{\infty}}+\|e(t)\|_{\LO{\infty}}+ \|i(t)\|_{\LO{\infty}}+\|r(t)\|_{\LO{\infty}} \big) < +\infty.
	\end{equation*}
	In particular, if  $\alpha(x,t)\leq \mu(x,t)$ for all $(x,t)\in \Omega\times[0,\infty)$, then the solution is bounded uniformly in time, i.e.
	\begin{equation*}
		\sup_{t>0}\big(\|s(t)\|_{\LO{\infty}}+ \|e(t)\|_{\LO{\infty}}+ \|i(t)\|_{\LO{\infty}}+ \|r(t)\|_{\LO{\infty}}\big) < +\infty.
	\end{equation*}
\end{theorem}

\appendix
\section{Proof of Lemma \ref{lem-L-infty-bound}}\label{appendix}
We first state the following interpolation inequality.
\begin{lemma}\label{lem-key2}
	Let $\gamma\geq 1$ and $\alpha>0$. Then we have the following interpolation inequality
	\begin{equation}\label{eq-key2-1}
		\|u\|^{\lambda}_{\LQ{\sigma}}\leq C \left( \|u\|^{\gamma}_{L^{\infty}(0,T;L^{\gamma}(\Omega))} + \|u\|^{\alpha}_{L^{\alpha}(0,T;L^{\beta}(\Omega))} \right),
	\end{equation}
	where
	\begin{equation}\label{eq-key2-2}
		\begin{cases}
			\lambda=\frac{\alpha\beta+\gamma(\beta-\alpha)}{2\beta-\alpha} \quad \text{and} \quad   \sigma=\frac{\alpha\beta+\gamma(\beta-\alpha)}{\beta}, &\text{ if } \quad 0<\beta<\infty,\\
			\lambda=\frac{\alpha+\gamma}{2} \quad \text{and} \quad   \sigma=\alpha+\gamma,   &\text{ if }\quad \beta=\infty.
		\end{cases}
	\end{equation}
\end{lemma}

\begin{proof}
	For $0<\beta<\infty$. We apply an interpolation inequality to get
	\begin{equation}\label{eq-key2-3}
		\|u\|_{\LQ{\sigma}}\leq \|u\|^{\theta}_{L^{\infty}(0,T;L^{\gamma}(\Omega))}  \|u\|^{1-\theta}_{L^{\alpha}(0,T;L^{\beta}(\Omega))}
	\end{equation}
	with $\theta\in(0,1)$ satisfying
	\begin{equation}
		\frac{1}{\sigma}=\frac{\theta}{\infty} + \frac{1-\theta}{\alpha}=\frac{\theta}{\gamma} + \frac{1-\theta}{\beta}.
	\end{equation}
	It follows that
	\begin{equation}\label{eq-key2-4}
		\theta=\frac{\gamma(\beta-\alpha)}{\alpha\beta+\gamma(\beta-\alpha)}, \quad
		1-\theta=\frac{\alpha\beta}{\alpha\beta + \gamma(\beta-\alpha)} \quad \text{and} \quad
		\sigma=\frac{\alpha\beta + \gamma(\beta-\alpha)}{\beta}.
	\end{equation}
	From \eqref{eq-key2-3} and $\lambda>0$ to be chosen, using Young's inequality we obtain
	\begin{equation*}
		\begin{split}
			\|u\|^{\lambda}_{\LQ{\sigma}} & \leq \|u\|^{\theta\lambda}_{L^{\infty}(0,T;L^{\gamma}(\Omega))} \|u\|^{(1-\theta)\lambda}_{L^{\alpha}(0,T;L^{\beta}(\Omega))}\\
			& \leq \frac{\theta\lambda}{\gamma} \|u\|^{\gamma}_{L^{\infty}(0,T;L^{\gamma}(\Omega))} + \frac{\gamma-\theta\lambda}{\gamma} \|u\|^{(1-\theta)\lambda\frac{\gamma}{\gamma-\theta\lambda}}_{L^{\alpha}(0,T;L^{\beta}(\Omega))}.
		\end{split}
	\end{equation*}
	Now, we choose $\lambda$ such that
	$$\frac{(1-\theta) \lambda \gamma}{\gamma-\theta\lambda}=\alpha,$$
	which is equivalent to
	$$\lambda=\frac{\alpha\gamma}{(1-\theta)\gamma+\theta\alpha}=\frac{\alpha\beta+\gamma(\beta-\alpha)}{2\beta-\alpha},$$
	where we used \eqref{eq-key2-4}.
	Therefore
	\begin{equation*}
		\|u\|^{\lambda}_{\LQ{\sigma}}\leq C \left(\|u\|^{\gamma}_{L^{\infty}(0,T;L^{\gamma}(\Omega))} + \|u\|^{\alpha}_{L^{\alpha}(0,T;L^{\beta}(\Omega))} \right)
	\end{equation*}
	with
	\begin{equation*}
		\lambda=\frac{\alpha\beta+\gamma(\beta-\alpha)}{2\beta-\alpha} \quad \text{and} \quad   \sigma=\frac{\alpha\beta+\gamma(\beta-\alpha)}{\beta}.
	\end{equation*}
	For $\beta=\infty$, we can use a similar method to get the result. This completes the proof of Lemma \ref{lem-key2}.
\end{proof}

\begin{lemma} [\cite{fellner2020global}, Lemma 2.4] \label{ap-k-lem1}
Let $\left\{y_{n}\right\}_{n \geq 1}$ be a sequence of positive numbers which satisfies
\begin{equation*}
y_{n+1} \leq K B^{n}\left(y_{n}^{\gamma}+y_{n}^{\kappa}\right)
\end{equation*}
where $K, B>0$ and $\gamma, \kappa>1$ are independent of $n$. Then there exists $\varepsilon>0$ such that, if $y_{1} \leq \varepsilon$, then
\begin{equation*}
\lim _{n \rightarrow \infty} y_{n}=0 .
\end{equation*}
\end{lemma}

We are now presenting a proof of Lemma \ref{lem-L-infty-bound} using the idea of Moser iteration.
\begin{proof}[Proof of Lemma \ref{lem-L-infty-bound}]
From \eqref{A5},
  \begin{equation*}
    f_i^{\eps}(u^{\eps})\leq f_i(u^{\eps})\leq F_i(u^{\eps}):=K_4\left(1+\sum^m_{j=1}(u_j^{\eps})^l\right).
  \end{equation*}
  Using Lemma \ref{lem-L-p-bound}, we can get for any $1\leq p<\infty$ a constant $C_p>0$ exists depending on $p$ (but not on $T$) such that
  \begin{equation}\label{eq-inb}
    \|f^{\eps}_i(u^{\eps})\|_{\LQ p}\leq C_p, \quad \forall i=1,\cdots,m.
  \end{equation}
   Let $k \geq 1$ be a constant which will be specified later. For each $j \geq 0$, we define
$$
v_j:=\left(u^{\eps}_i-k+\frac{k}{2^j}\right)_{+}=\max \left\{u^{\eps}_i-k+\frac{k}{2^j},  0\right\}
$$
and
$$
A_j:=\left\{(x, t) \in Q_T: u^{\eps}_i(x, t) \geq k-\frac{k}{2^j}\right\} .
$$
The following simple observations will be helpful
\begin{equation}\label{eq-inf-1}
  \begin{aligned}
&v_{j+1}(x, t) \leq v_j(x, t), &&   \forall (x, t) \in A_j, \\
&v_j(x, t)  \geq \frac{k}{2^{j+1}}, && \forall (x, t) \in A_{j+1} \subset A_j.
\end{aligned}
\end{equation}
 By multiplying the equation $\partial_t u^{\eps}_i - \Delta\left(D_i(x,t) \,\Phi(u^{\eps})\, \nabla u^{\eps}_i\right)=f^{\eps}_i$ by $v_{j+1}$ and integrating on $Q_T$, we have
\begin{equation*}
\begin{aligned}
& \sup _{t \in(0, T)}\|v_{j+1}(t)\|^2_{\LO2} + 2 \lambda \int_0^T \int_{\Omega}  \Phi(u^{\eps}) \left|\nabla v_{j+1}\right|^2 \,d x d t \\
& \quad \leq\|v_{j+1}(0)\|^2_{\LO2} + 2 \int_0^T \int_{\Omega} f^{\eps}_i v_{j+1} \,d x d t,
\end{aligned}
\end{equation*}
where we used \eqref{D:cond-1}.
Using \eqref{Q2}, we have
\begin{equation}\label{eq-inf-2}
\begin{aligned}
& \sup _{t \in(0, T)}\|v_{j+1}(t)\|^2_{\LO2} + 2 \lambda M_i \int_0^T \int_{\Omega}(u^{\eps}_i)^{b}\left|\nabla v_{j+1}\right|^2 \,d x d t \\
& \quad \leq\|v_{j+1}(0)\|^2_{\LO2}+2 \int_0^T \int_{\Omega} f^{\eps}_i v_{j+1} \,d x d t .
\end{aligned}
\end{equation}
Note that $u^{\eps}_i \geq k-\frac{k}{2^j} \geq \frac{k}{2}$ on $A_j$ ($j\geq1$), we have
\begin{equation*}
  \begin{aligned}
2\lambda M_i \int_0^T \int_{\Omega}(u^{\eps}_i)^{b}\left|\nabla v_{j+1}\right|^2 \,d x d t & \geq 2\lambda M_i \iint_{A_j}(u^{\eps}_i)^{b}\left|\nabla v_{j+1}\right|^2 \,d x d t \\
& \geq 2\lambda M_i \left(\frac{k}{2}\right)^{b} \iint_{A_j}\left|\nabla v_{j+1}\right|^2 \,dx dt \\
& \geq \frac{\lambda M_i}{2^{b-1}} \int_0^T \int_{\Omega}\left|\nabla v_{j+1}\right|^2 \,dxdt,
\end{aligned}
\end{equation*}
thanks to $k\geq1$  and the fact that $v_{j+1} \equiv 0$ on $Q_T \backslash A_{j+1} \supset Q_T \backslash A_j$ since $A_{j+1} \subset A_j$.
By adding $\frac{\lambda M_i}{2^{b-1}} \int_0^T\|v_{j+1}\|^2_{\LO2} d t$ to both sides of \eqref{eq-inf-2}, we get
\begin{equation*}
  \begin{aligned}
& \sup _{t \in(0, T)}\|v_{j+1}(t)\|^2_{\LO2} + \frac{\lambda M_i}{2^{b-1}} \int_0^T\|v_{j+1}\|_{H^1(\Omega)}^2 \, d t \\
& \leq \frac{\lambda M_i}{2^{b-1}}  \int_0^T\|v_{j+1}\|^2_{\LO2}\, d t+\|v_{j+1}(0)\|^2_{\LO2} + 2\int_0^T \int_{\Omega} f^{\eps}_i v_{j+1} \,d x d t,
  \end{aligned}
\end{equation*}
which yields
\begin{equation}\label{eq-inf-3}
\|v_{j+1}\|_{W(0, T)}^2 \leq C\left(\int_0^T\|v_{j+1}\|^2_{\LO2} \, d t+\|v_{j+1}(0)\|^2_{\LO2} + \int_0^T \int_{\Omega} f^{\eps}_i v_{j+1} \,d x d t \right).
\end{equation}
By definition, when we choose $k \geq 2\|u^{\eps}_0\|_{L^{\infty}(\Omega)}$, we have
\begin{equation}\label{eq-inf-4}
\|v_{j+1}(0)\|^2_{\LO2}=\|\left(u^{\eps}_0-k+\frac{k}{2^{j+1}}\right)_{+}\|^2_{\LO2}=0.
\end{equation}
By using \eqref{eq-inf-1}, we have with $1 \leq \frac{2^{j+1}}{k} v_j$ on $A_{j+1}$
\begin{equation}\label{eq-inf-5}
\begin{aligned}
\int_0^T \int_{\Omega}\left|v_{j+1}\right|^2 d x d t & =\int_0^T \int_{\Omega} \mathbf{1}_{A_{j+1}}\left|v_{j+1}\right|^2 \,d x d t \\
& \leq \int_0^T \int_{\Omega} \mathbf{1}_{A_{j+1}}\left|v_j\right|^2 \,d x d t \\
& \leq\left(\frac{2^{j+1}}{k}\right)^{\frac{4}{N}} \int_0^T \int_{\Omega} \mathbf{1}_{A_{j+1}}\left|v_j\right|^{2+\frac{4}{N}} \,d x d t \\
& \leq C\left(2^{\frac{4}{N}}\right)^j\|v_j\|_{W(0, T)}^{2+\frac{4}{N}}.
\end{aligned}
\end{equation}
Choose $p>\frac{N+2}{2}$, we have
\begin{equation*}
  \sigma:=\frac{p-1}{p}\left(2+\frac{4}{N}\right)>2.
\end{equation*}
Moreover,
\begin{equation*}
  \frac{\sigma p}{p-1}=2+\frac{4}{N},
\end{equation*}
implying
\begin{equation*}
  \|v_j\|_{L^{\frac{\sigma p}{p-1}}\left(Q_T\right)} \leq C\|v_j\|_{W(0, T)}.
\end{equation*}
We now can use H\"{o}lder's inequality to estimate with \eqref{eq-inf-1}
\begin{equation}\label{eq-inf-6}
  \begin{aligned}
\int_0^T \int_{\Omega} f^{\eps}_i v_{j+1}  \,d x d t & \leq \int_0^T \int_{\Omega} f^{\eps}_i v_{j+1}\left(\frac{2^{j+1}}{k}\right)^{\sigma-1} v_j^{\sigma-1} \,d x d t \\
& \leq\left(\frac{2^{j+1}}{k}\right)^{\sigma-1} \int_0^T \int_{\Omega}\left|f^{\eps}_i \| v_j\right|^\sigma \,d x d t \\
& \leq C\left(2^{\sigma-1}\right)^j\|f^{\eps}_i\|_{L^p\left(Q_T\right)}\|v_j\|^\sigma_{L^{\frac{\sigma p}{p-1}} \left(Q_T\right)} \\
& \leq C\left(2^{\sigma-1}\right)^j\|f^{\eps}_i\|_{L^p\left(Q_T\right)}\|v_j\|_{W(0, T)}^\sigma .
\end{aligned}
\end{equation}
Inserting \eqref{eq-inf-4}, \eqref{eq-inf-5} and \eqref{eq-inf-6} into \eqref{eq-inf-3} leads to
\begin{equation}\label{eq-inf-7}
\|v_{j+1}\|_{W(0, T)}^2 \leq C\left(1+\|f^{\eps}_i\|_{L^p\left(Q_T\right)}\right) B^j\left(\|v_j\|_{W(0, T)}^{2+\frac{4}{N}}+\|v_j\|_{W(0, T)}^\sigma\right)
\end{equation}
for all $j \geq 0$, where $B=\max \left\{2^{\frac{4 }{N}} , 2^{\sigma-1}\right\}$.\\
By setting $Y_j=\|v_j\|_{W(0, T)}^2$, we obtain a sequence $\left\{Y_n\right\}_{n \geq 1}$ satisfying the property in Lemma \ref{ap-k-lem1}. It remains to show that $Y_1$ is small enough.\\
We show now that for any $\eta>0$, there exists $k \geq \max \left\{1, 2\|u^{\eps}_0\|_{L^{\infty}(\Omega)}\right\}$ large enough such that
\begin{equation}\label{eq-inf-8}
Y_1=\|v_1\|^2_{W(0, T)} \leq \eta.
\end{equation}
 Let $p> 1$. By multiplying \eqref{appro-system} by $p |u_i^{\eps}|^{p-1}$ and integrating over $\Omega$, we obtain
\begin{equation}\label{eq-e1}
\begin{split}
  \frac{d}{dt}\|u_i^{\eps}\|^{p}_{\LO{p}} -p\int_{\Omega}\nabla\cdot(D(x,t) \Phi(u^{\eps}) \nabla u_i^{\eps}) |u_i^{\eps}|^{p-1} \,dx =p\int_{\Omega}f_i^{\eps} |u_i^{\eps}|^{p-1} \,dx.
\end{split}
\end{equation}
Integration by parts and the homogeneous Neumann boundary condition $(D(x,t) \Phi(u^{\eps}) \nabla u_i^{\eps})\cdot \nu= 0$ lead to
    \begin{equation*}
    \begin{aligned}
     -p\int_{\Omega}&\nabla\cdot(D(x,t) \Phi(u^{\eps})  \nabla u_i^{\eps}) |u_i^{\eps}|^{p-1}\,dx\\
     &\geq \lambda M_ip(p-1)\int_{\Omega}|u_i^{\eps}|^{b+p-2}|\nabla u_i^{\eps}|^2 \,dx\\
     &=\lambda M_ip(p-1)\frac{4}{(b+p)^2}\int_{\Omega}|\nabla \left((u_i^{\eps})^{\frac{b+p}{2}}\right)|^2 \,dx\\
     &=:C(p,\lambda, M_i)\int_{\Omega}|\nabla \left((u_i^{\eps})^{\frac{b+p}{2}}\right)|^2 \,dx,
	\end{aligned}
	\end{equation*}
where we used \eqref{D:cond-1} and \eqref{Q2}.
  By H\"{o}lder's inequality
		\begin{equation*}
		\left|p\int_{\Omega}f_i^{\eps} |u_i^{\eps}|^{p-1} \,dx\right| \leq p\|f_i^{\eps}\|_{\LO p}\|u_i^{\eps}\|_{\LO{p}}^{p-1}.
		\end{equation*}
		Therefore, it follows from \eqref{eq-e1} that
		\begin{equation}\label{eq-mu-ori}
		\frac{d}{dt}\|u_i^{\eps}\|_{\LO{p}}^{p} + C(p,\lambda, M_i) \int_{\Omega}\left|\nabla\left((u_i^{\eps})^{\frac{b+p}{2}}\right)\right|^2dx \leq p\|f_i^{\eps}\|_{\LO {p}}\|u_i^{\eps}\|_{\LO{p}}^{p-1}.
		\end{equation}		
		By applying for $r<1$ the elementary inequality
		\begin{equation*}
		y' \leq \alpha(t)y^{1-r} \quad \Longrightarrow \quad y(T) \leq \left[y(0)^{r}+ r \int_0^T\alpha(t) \,dt\right]^{1/r},
		\end{equation*}
		to \eqref{eq-mu-ori} with $r=1/p$ and $y(t) = \|u_i^{\eps}(t)\|_{\LO p}^{p}$, we obtain
		\begin{equation}\label{eq-CT}
         \begin{split}
           \|u_i^{\eps}(T)\|_{\LO p}^{p} &\leq \left[\|u^{\eps}_0\|_{\LO p} + \int_0^{T}\|f_i^{\eps}\|_{\LO p} \,dt \right]^{ p} \\
           &\leq \left[\|u^{\eps}_0\|_{\LO p} + \|f_i^{\eps}\|_{\LQ{p}}T^{\frac{(p-1)}{p}} \right]^{p} =: C_{T}.
         \end{split}
		\end{equation}
		That means
		\begin{equation}\label{eq-p-1}
		u_i^{\eps} \in L^{\infty}(0,T;L^{p}(\Omega)) \quad \text{ and } \quad \|u_i^{\eps}(T)\|_{\LO p}^{p} \leq C_{T}
		\end{equation}
		where $C_{T}$ defined in \eqref{eq-CT} grows at most polynomially in $T$. By integrating \eqref{eq-mu-ori} with respect to $t$ on $(0,T)$ and by using Young's inequality and the convention $r := b+p\geq p>1$, we get
		\begin{equation*}
		\begin{aligned}
		C(p,\lambda, M_i)\int^T_0\intO\left|\nabla\left((u_i^{\eps})^{\frac{r}{2}}\right)\right|^2 \,dxdt &\leq \|u_0^{\eps}\|_{\LO p}^{p} +
        p\int_0^T\|f_i^{\eps}\|_{\LO p}\|u_i^{\eps}\|_{\LO p}^{p-1} \,dt\\
		&\leq \|u_0^{\eps}\|_{\LO p}^{p} + p\|f_i^{\eps}\|_{L^{p}(Q_T)}\|u_i^{\eps}\|_{L^{p}(Q_T)}^{p-1}.
		\end{aligned}
		\end{equation*}
		By adding $C(p,\lambda, M_i)\int^T_0\intO (u_i^{\eps})^{r}  \,d x d t$ to both sides, we have
		\begin{equation}\label{eq-h0}
		\begin{aligned}	
        C(p,\lambda, M_i)&\int_0^T\|(u_i^{\eps})^{\frac{r}{2}}\|_{H^1(\Omega)}^2 \,dt\\
        & = C(p,\lambda, M_i)\int_0^T\left[
		\int_{\Omega}\left|\nabla\left((u_i^{\eps})^{\frac{r}{2}}\right)\right|^2 \,dx + \int_{\Omega}\left| (u_i^{\eps})^{\frac{r}{2}} \right|^2  \,dx\right]dt\\
		&\leq \|u_0^{\eps}\|_{\LO p}^{p} + p\|f_i^{\eps}\|_{L^{p}(Q_T)}\|u_i^{\eps}\|_{L^{p}(Q_T)}^{p-1} + C(p,\lambda, M_i)\int_0^T\|u_i^{\eps}\|_{\LO {r}}^{r} \,dt.
		\end{aligned}
		\end{equation}			
		By the Sobolev's embedding, we have
		\begin{equation}\label{eq-h1}
		C(p,\lambda, M_i)\int_0^T\|(u_i^{\eps})^{\frac{r}{2}}\|^2_{H^1(\Omega)}\,dt \geq C(p,\lambda, M_i) \,C_S^2\int_0^T
\|u_i^{\eps}\|_{\LO s}^{r} \,dt
		\end{equation}
with
        \begin{equation*}
		s =
        \begin{cases}
        \frac{r N}{N-2} &\text{ if } N \geq 3,\\
       r < s < \infty \text{ arbitrary } &\text{ if } N = 1,2.
        \end{cases}
		\end{equation*}
		On the other hand, by using the bound $\|u_i^{\eps}(t)\|_{\LO p}^{p} \leq C_{T}$ in \eqref{eq-p-1} and the interpolation inequality
		\begin{equation*}
		\|u_i^{\eps}\|_{\LO {r}} \leq \|u_i^{\eps}\|_{\LO p}^{\alpha}\|u_i^{\eps}\|_{\LO s}^{1-\alpha}
           \leq C_{T}^{\frac{\alpha}{p}}\|u_i^{\eps}\|_{\LO s}^{1-\alpha}
		\end{equation*}
      with
       \begin{equation*}
	   \frac{1}{r} = \frac{\alpha}{p} + \frac{1-\alpha}{s} \quad \text{and} \quad \alpha = \frac{2p}{2p+bN}\in(0,1],
		\end{equation*}
	we estimate {in the cases $b>0$ for which $\alpha<1$}
		\begin{equation}\label{eq-h2}
        \begin{split}
          C(p,\lambda, M_i)\int_0^T\|u_i^{\eps}\|_{\LO {r}}^{r} \,dt &\leq C(p,\lambda, M_i )\int_{0}^{T}C_{T}^{\frac{\alpha r}{p}} \|u_i^{\eps}\|_{\LO s}^{(1-\alpha)r} \,dt \\
          &\leq \frac{C(p,\lambda, M_i)\,C_S^2}{2}\int_0^T\|u_i^{\eps}\|_{\LO s}^{r} \,dt + CC_{T}^{\frac{r}{p}}T,
        \end{split}
		\end{equation}
		where we have used Young's inequality (with the exponents $1=(1-\alpha) + \alpha$) in the last step. {Note that if $b=0$, the bound \eqref{eq-h2} still holds true without the first term and with $\frac{r}{p}=1$.}
		Inserting \eqref{eq-h1} and \eqref{eq-h2} into \eqref{eq-h0} leads to
		\begin{equation}\label{eq-DT}
		\begin{aligned}
		\int_0^T\|u_i^{\eps}\|_{\LO s}^{r}dt &\leq \frac{2}{C(p,\lambda, M_i)\,C_S^2}\left[\|u_0^{\eps}\|_{\LO p}^{p} + p \|f_i^{\eps}\|_{L^{p}(Q_T)}\|u_i^{\eps}\|_{L^{p}(Q_T)}^{p-1}  + CC_{T,0}^{\frac{r}{p}}T\right]\\
		&\leq \frac{2}{C(p,\lambda, M_i)\,C_S^2}\left[\|u_0^{\eps}\|_{\LO p}^{p} + p\|f_i^{\eps}\|_{L^{p}(Q_T)} C_{T}^{\frac{p-1}{p}}  + CC_{T}^{\frac{r}{p}}T\right] \\
        &=: D_{T} \quad (\text{use (\ref{eq-p-1}))}.
		\end{aligned}
		\end{equation}		
		It follows that
		\begin{equation*}
		u^{\eps}_i \in L^{r}(0,T;L^{s}(\Omega))
		\end{equation*}
        with
        \begin{equation*}
		\begin{cases}
        s = \frac{r N}{N-2} &\text{ if } N\geq 3,\\
		r< s <\infty \text{ arbitrary } &\text{ if } N=1,2
		\end{cases}
		\end{equation*}
		and
		\begin{equation*}
		\int_0^T\|u_i^{\eps}\|_{\LO s}^{r} \,dt \leq D_{T}
		\end{equation*}
		with $D_{T}$ defined in \eqref{eq-DT}.\\ 		
Thus, we have
\begin{equation*}
 \|u^{\eps}_i\|_{L^{\infty}\left(0, T ; L^p(\Omega)\right)}+\|u^{\eps}_i\|_{L^{r}\left(0, T ; L^s(\Omega)\right)} \leq C_T,
\end{equation*}
where $r=b+p \geq p$ and $s=\frac{r N}{N-2}$ if $N \geq 3$ and $r<s<+\infty$ arbitrary if $N \leq 2$.\\
Using Lemma \ref{lem-key2}, we have
\begin{equation*}
  \|u^{\eps}_i\|_{L^{\tau}\left(Q_T\right)} \leq C_T \quad \text { with } \quad \tau=
  \begin{cases}
  \frac{N r +2 p}{N} & \text { if } N \geq 3, \\
  <r +p \text { arbitrary } & \text { if } N \leq 2 .
  \end{cases}
\end{equation*}
Direct calculations show that $\tau>2+\frac{4}{N}$ if $N \geq 2$ and $\tau>3$ if $N=1$. In particular,
\begin{equation}\label{eq-inf-9}
\|u^{\eps}_i\|_{L^{2+\frac{4}{N}}\left(Q_T\right)} \leq C_T \text { for } d \geq 2 \text { and }\|u^{\eps}_i\|_{L^3\left(Q_T\right)} \leq C_T \text { for } d=1 .
\end{equation}
From \eqref{eq-inf-3},
\begin{equation}\label{eq-inf-10}
\|v_1\|_{W(0, T)}^2 \leq C\left(\int_0^T\|v_1(t)\|^2_{\LO2} \,d t+\|v_1(0)\|^2_{\LO2} + \int_0^T \int_{\Omega} f^{\eps}_i v_1 \,dxdt\right).
\end{equation}
Since $k \geq 2\|u^{\eps}_0\|_{L^{\infty}(\Omega)},\|v_1(0)\|^2_{\LO2}=\|\left(u^{\eps}_0-\frac{k}{2}\right)_+\|^2_{\LO2}=0$.\\
Consider now the case $N \geq 2$. By using \eqref{eq-inf-1}, it yields
\begin{equation}\label{eq-inf-11}
\begin{aligned}
\int_0^T \int_{\Omega}\left|v_1\right|^2 \,d x d t & =\int_0^T \int_{\Omega} \mathbf{1}_{A_1}\left|v_1\right|^2 \,d x d t \leq\left(\frac{4}{k}\right)^{\frac{4}{N}} \int_0^T \int_{\Omega}\left|v_0\right|^{2+\frac{4}{N}} \,d x d t \\
& \leq\left(\frac{4}{k}\right)^{\frac{4}{N}}\|u^{\eps}_i\|_{L^{2+\frac{4}{N}}(Q_T)}^{2+\frac{4}{N}} \leq\left(\frac{4}{k}\right)^{\frac{4}{N}} C_T,
\end{aligned}
\end{equation}
recalling that $v_0=(u^{\eps}_i)_{+}$. Similarly to \eqref{eq-inf-6}, we get
\begin{equation}\label{eq-inf-12}
\int_0^T \int_{\Omega} f^{\eps}_i v_1 \,d x d t \leq\left(\frac{4}{k}\right)^{\sigma-1}\|f^{\eps}_i\|_{L^p\left(Q_T\right)}\|u^{\eps}_i \|_{L^{2+\frac{4}{N}}\left(Q_T\right)}^\sigma \leq\left(\frac{4}{k}\right)^{\sigma-1} C_T .
\end{equation}
From \eqref{eq-inf-10}, \eqref{eq-inf-11} and \eqref{eq-inf-12}, we get \eqref{eq-inf-8} if
$$
k=4 \max \left\{\left(\frac{C_T}{\eta}\right)^{\frac{N}{4}}, \left(\frac{C_T}{\eta}\right)^{\frac{1}{\sigma-1}}\right\}.
$$
Thus, with this choice of $k$, it follows that
$$
0=\lim _{j \rightarrow \infty} Y_j=\|(u^{\eps}_i-k)_{+}\|^2
$$
and hence,
$$
\|u^{\eps}_i\|_{L^{\infty}\left(Q_T\right)} \leq k=4 \max \left\{\left(\frac{C_T}{\eta}\right)^{\frac{N}{4}}, \left(\frac{C_T}{\eta}\right)^{\frac{1}{\sigma-1}}\right\}
$$
which is our desired estimate.\\
The proof for the case $N=1$ is very similar using
$$
\int_0^T \int_{\Omega}\left|v_1\right|^2 \,d x d t \leq \frac{4}{k} \int_0^T \int_{\Omega}\left|v_0\right|^3 \,d x d t \leq \frac{4}{k} C_T
$$
and
$$
\int_0^T \int_{\Omega} f^{\eps}_i v_1 \,d x d t \leq\left(\frac{4}{k}\right)^{\frac{4 \xi}{1+2 \xi}}\|f^{\eps}_i\|_{L^p\left(Q_T\right)}\|u^{\eps}_i\|_{L^3(Q_T)}^{1+\frac{4 \xi}{1+2 \xi}} \leq\left(\frac{4}{k}\right)^{\frac{4 \xi}{1+2 \xi}} C_T
$$
where $\xi=\frac{1}{2}(2p-3)>0$. We therefore omit the details.\\
Thus, we  completed the proof of Lemma \ref{lem-L-infty-bound}.
\end{proof}

\noindent{\bf Acknowledgement.} The authors would like to thank Prof. Pierluigi Colli, Prof. Gabriela Marinoschi, and Prof. Elisabetta Rocca for their inspired and fruitful discussion.

B.Q. Tang received funding from the FWF project ``Quasi-steady-state approximation for PDE'', number I-5213. J. Yang are partially supported by NSFC Grants
No. 12271227 and China Scholarship Council (Contract No. 202206180025).

\newcommand{\etalchar}[1]{$^{#1}$}

\end{document}